\documentclass{amsart}
\usepackage{amsmath}
\usepackage{amsthm}
\usepackage{amssymb}
\usepackage{amsfonts}
\usepackage{epsfig}
\usepackage{setspace}
\usepackage{color}
\usepackage{amscd}
\usepackage[enableskew]{youngtab}
\newtheorem{theorem}{Theorem}[section]

\newtheorem{teo}[theorem]{Theorem}
\newtheorem{lm}[theorem]{Lemma}
\newtheorem{tr}[theorem]{Theorem}

\newtheorem{cor}[theorem]{Corollary}
\newtheorem{df}[theorem]{Definition}
\newtheorem{rem}[theorem]{Remark}
\newtheorem{ex}[theorem]{Example}
\newtheorem{pr}[theorem]{Proposition}

\newcommand{\la}{\lambda}

\begin{document}
\title[Even-primitive vectors in induced supermodules]{Even-primitive vectors in induced supermodules for general linear supergroups and in costandard supermodules for Schur superalgebras}

\author{Franti\v sek~ Marko}

\email{fxm13@psu.edu}

\address{Penn State Hazleton, 76 University Drive, Hazleton, PA 18202, USA}

\begin{abstract} Let $G=GL(m|n)$ be the general linear supergroup over an algebraically closed field $K$ of characteristic zero
and let $G_{ev}=GL(m)\times GL(n)$  be its even subsupergroup. The induced supermodule $H^0_G(\la)$, corresponding to a dominant weight $\la$ of $G$,  
can be represented as $H^0_{G_{ev}}(\la)\otimes \Lambda(Y)$, 
where $Y=V_m^*\otimes V_n$ is a tensor product of the dual of the natural $GL(m)$-module $V_m$ and the natural $GL(n)$-module $V_n$, and $\Lambda(Y)$ is 
the exterior algebra of $Y$.
For a dominant weight $\la$ of $G$, we construct explicit $G_{ev}$-primitive vectors in $H^0_G(\la)$. Related to this, we give explicit formulas for $G_{ev}$-primitive vectors of the supermodules $H^0_{G_{ev}}(\la)\otimes \otimes^k Y$.
Finally, we describe a basis of $G_{ev}$-primitive vectors in the largest polynomial subsupermodule $\nabla(\la)$ of $H^0_G(\la)$ (and therefore in the costandard supermodule of the corresponding Schur superalgebra $S(m|n)$). This yields a description of a basis of $G_{ev}$-primitive vectors in arbitrary induced supermodule $H^0_G(\la)$.
\end{abstract}

\maketitle
\section*{Introduction}

Throughout the paper, let $G=GL(m|n)$ be the general linear supergroup defined over an algebraically closed field $K$  of characteristic zero and 
$G_{ev}=GL(m)\times GL(n)$ be its even subsupergroup.

In our earlier paper \cite{fm}, we gave explicit formulae for certain $G_{ev}$-primitive (or shortly, even-primitive) vectors $\pi_{I|J}$. 
These vectors form a basis of even-primitive vectors of $H^0_G(\la)$ in some special weight spaces (which are described using the concept of robustness) - see Theorem 4.4 of \cite{fm}.

The main obstacle is that vectors $\pi_{I|J}$ do not belong to $H^0_G(\la)$ in general. In the comments following Theorem 4.4 of \cite{fm} we have discussed the possibility that all even-primitive vectors in $H^0_G(\la)$ could be written as a certain linear combination of vectors $\pi_{I|J}$. 
One of the purposes of this paper is to confirm this speculation for even-primitive vectors in the largest polynomial subsupermodule $\nabla(\la)$ of $H^0_G(\la)$.
We also give explicit formulae for certain even-primitive vectors in induced $G$-supermodules $H^0_G(\la)$ and describe an explicit basis of even-primitive vectors of 
$\nabla(\la)$.
Since, for polynomial weights $\la$, the multiplicity of even-primitive vectors is given by certain Littlewood-Richardson coefficients, one can view the results of this paper as an "algebraisation" of these combinatorial quantities.

The combinatorial techniques which we use are related to Young tableaux, and are appropriate for the description of $\nabla(\la)$. The supermodules $\nabla(\la)$ are costandard supermodules in a Schur superalgebra $S(m|n,r)$ of an  appropriate degree $r$, which are of independent interest, and our results bring an understanding of their $G_{ev}$-structure.  Results about even-primitive vectors were previously known only for special cases of Schur superalgebras $S(1|1)$, $S(2|1)$, $S(3|1)$ and $S(2|2)$, and were given in papers \cite{mz0,gmz,gm,fm1}. Once all even-primitive vectors in $\nabla(\la)$ are known, we can tensor with appropriate powers of "even" determinants and obtain a description of even-primitive vectors for arbitraty induced supermodule $H^0_G(\la)$.

Understanding of certain even-primitive vectors in induced supermodules $H^0_G(\la)$ was one of the ingredients used in the proof of the linkage principle for $GL(m|n)$ in arbitrary characteristic $p$ different from $2$ (it was recently proved in \cite{mz2}). The description of even-primitive vectors in modules $\nabla(\la)$ is likely connected to the linkage principle for Schur superalgebras. In a forthcoming paper \cite{fm3}, we applied this description of even-primitive vectors to derive results related to odd linkage for $G$. This bring a combinatorial perspective and provide additional valuable insight into representation theory of $GL(m|n)$.

The structure of the paper is as follows.
In Section 1 we fix notation related to the general linear supergroup $G$ and induced supermodules $H^0_G(\la)$ and $H^0_{G_{ev}}(\la)$. In particular, we work with the induced supermodule $H^0_G(\la)$ represented as $H^0_{G_{ev}}(\la)\otimes \Lambda(Y)$, 
where $Y=V_m^*\otimes V_n$ is the tensor product of the dual of the natural $GL(m)$-module $V_m$ and the natural $GL(n)$-module $V_n$, and $\Lambda(Y)$ is 
the exterior algebra of $Y$. 
Also, we describe even-primitive elements $\pi_{I|J}$ and their building blocks $\rho_{I|J}$.
In Section 2 we derive congruences for some transposition operators involving $\rho_{I|J}$ modulo certain bideterminants. In Section 3 we prove numerous auxiliary determinantal identities used in the following sections.
In Section 4 we explain general setup of diagrams and tableaux, define operators $\sigma^+$, $\sigma^-$, $\sigma$ and positioning maps. Using these tools we 
construct $G_{ev}$-primitive vectors in $H^0_{G_{ev}}(\la)\otimes \otimes^k Y$ and in the floors $F_k=H^0_{G_{ev}}(\la)\otimes \wedge^k Y$ of $H^0_G(\la)$.
In Section 5 we define operators $\tau^+$, $\tau^-$ and $\tau$ on tableaux $T$. We also discuss properties of these operators and of repositioning maps.
Using these operators we are able to describe certain even-primitive vectors of $H^0_G(\la)$. In Section 6 we disscuss Clausen preorders, linear independence of even-primitive vectors and action of the operator $\tau$ on Littlewood-Richardson tableaux. Also, for hook partition $\la$ and irreducible $H^0_G(\la)$, we obtain a basis 
of even-primitive vectors in $H^0_G(\la)$.
In Section 7 we consider Schur superalgebra $S(m|n)$ and determine a basis of all even-primitive vectors in $\nabla(\la)$, the largest polynomial subsupermodule of $H^0_G(\la)$, and of the corresponding costandard supermodule for $S(m|n)$. 
We also provide a connection of our combinatorial construction to pictures in the sense of Zelevinski \cite{zel}.

\section{Background and notation}

For more information about the general linear supergroup $G=GL(m|n)$, its even subsupergroup $G_{ev}$ and their coordinate and distribution algebras, superderivations, simple and induced supermodules within the context relevant to this paper see \cite{z,fm,bk}.

\subsection{General linear supergroups}
Let the parity of the index $1\leq i\leq m$ be $|i|=0$ and the parity of the index $m+1\leq j\leq n$ be $|j|=1$.  
Let $A(m|n)$ be the superalgebra freely generated by elements $c_{ij}$ for $1\leq i,j \leq m+n$
subject to the supercommutativity relation \[c_{ij}c_{kl}=(-1)^{|c_{ij}||c_{kl}|} c_{kl}c_{ij},\]
where $|c_{ij}|\equiv |i|+|j| \pmod 2$ is the parity of the element $c_{ij}$.
Denote by $A(m|n)_0$ and $A(m|n)_1$, respectively, the subsets of $A(m|n)$ consisting of elements of even and odd parity, respectively. 
There is a natural grading on $A(m|n)$ given by degree $r$. The homogeneous component of $A(m|n)$ corresponding to degree $r\geq 0$ will be 
denoted by $A(m|n,r)$. The dual of $A(m|n,r)$ is the Schur superalgebra denoted by $S(m|n,r)$.

Denote by $A_{ev}(m|n)$ the subsuperalgebra of $A(m|n)$ spanned by the elements $c_{ij}$ such that $|i|=|j|$.
The set $A(m|n)_0$ is a subsuperalgebra of $A(m|n)$ but it is not a domain. On the other hand, the superalgebra
$A_{ev}(m|n)$ is a domain. We will work inside the localization $A(m|n)(A_{ev}(m|n)\setminus 0)^{-1}$ that will be denoted by $K(m|n)$. 

The superalgebra $A(m|n)$ also has a structure of a superbialgebra given by 
comultiplication $\Delta(c_{ij})=\sum_{k=1}^{m+n} c_{ik}\otimes c_{kj}$   
and the counit $\epsilon$ given by $\epsilon(c_{ij})=\delta_{ij}$.

Write the $(m+n)\times (m+n)$-matrix $C=(c_{ij})$ as a block matrix 
\begin{equation*}
C=
\begin{pmatrix}
C_{11} & C_{12} \\ 
C_{21} & C_{22}
\end{pmatrix},
\end{equation*}
where $C_{11}, C_{12}, C_{21}$ and $C_{22}$ are matrices of sizes $m\times m$, $m\times n$, $n\times m$ and $n\times n$, respectively. 
The localization of $A(m|n)$ at the element $det(C_{11})\, det(C_{22})$ is a Hopf superalgebra $K[G]$, where the antipode
$s:X\to X'$ is given by 

\[\begin{aligned}&X'_{11}=(X_11-X_{12}X_{22}X_{21})^{-1}, &X'_{22}&=(X_{22}-X_{21}X_{11}^{-1}X_{12})^{-1}\\
&X'_{12}=-X_{11}^{-1}X_{12}X'_{22}, &X'_{21}&=-X_{22}^{-1}X_{21}X'_{11}
\end{aligned}\]

The Hopf superalgebra $K[G]$ is a coordinate algebra of the general linear supergroup $G=GL(m|n)$.
The general linear group $G$ is a functor from commutative superalgebras to groups 
given by $A\mapsto Hom_{superalg}(K[G],A)$. The category of $G$-supermodules is identical to the category of $K[G]$-supercomodules. 

General linear groups $GL(m)$ and $GL(n)$ are embbeded in $GL(m|n)$ as its even subsupergroups. 
There is a standard maximal even subsupergroup $G_{ev}$ of $G$ such that $G_{ev}\simeq GL(m)\times GL(n)$, corresponding to matrices $X$ where blocks $X_{12}$ and $X_{21}$ vanish.

The supergroups $G$ and $G_{ev}$ have the same standard maximal torus $T=T(m|n)\simeq (K^*)^{m+n}$ corresponding to diagonal matrices. 
Therefore, weights of $G$ and $G_{ev}$ are the same, and a weight $\la$ will be denoted by $(\la_1, \ldots, \la_m|\la_{m+1}, \ldots \la_{m+n})$ and its degree
$\sum_{k=1}^{m+n} \la_k$ will be denoted by $|\la|$.

Simple $G$- and $G_{ev}$-supermodules are in one-to-one correspondence (up to a parity shift) with dominant weights $\la$, where $\la_1\geq \ldots \geq \la_m$ and $\la_{m+1} \geq \ldots \la_{m+n}$. We will disregard the parity shift and denote the simple $G$-supermodule of the highest weight $\la$ by $L(\la)$, and the simple $G_{ev}$-module of the highest weight $\la$ by $L_{ev}(\la)$. 

Since the characteristic of the field $K$ is zero, the $G$-module structure is described by the action of Lie superalgebra $\mathfrak{gl}(m|n)$.
In the paper \cite{fm} we have used the language of superderivations. The action of $e_{ji}\in \mathfrak{gl}(m|n)$ corresponds to the action of right superderivation 
$_{ij}D$ so that $e_{ji}.v = (v)_{ij}D$.

The superderivation $_{ij}D$ is determined by the property 
\[(uv)_{ij}D=(-1)^{(|i|+|j|)|v|}(u)_{ij}Dv + u (v)_{ij}D\]
and 
its action on elements of $A(m|n)$ given by $(c_{kl})_{ij}D=\delta_{li} c_{kj}$. 
The action of $_{ij}D$ extends to $K(m|n)$ using the quotient rule
\[(\frac{u}{v})_{ij}D=\frac{(u)_{ij}Dv-u(v)_{ij}D}{v^2}\] 
for $u,v\in A(m|n)$ and $v$ even.
For more information, consult Section 4 of \cite{lz}.

\subsection{Induced $GL(m)$-modules}

We need to review some classical results about induced modules for general linear superalgebra $GL(m)$.

The group $GL(m)$ has a standard torus $T^+\simeq (K^*)^{m}$ corresponding to diagonal matrices and a standard Borel subgroup $B^+$ corresponding to lower triangular matrices of size
$m\times m$. Let $\la^+=(\la_1, \ldots, \la_m)$ be a weight of $GL(m)$ and let $K_{\la^+}$ be the one-dimensional $T^+$-module of highest weight $\la^+$. 
We can consider $K_{\la^+}$ as a $B^+$-module via extending the action from $T^+$ to $B^+$ trivially. 
The induced supermodule $H^0_{GL(m)}(\la^+)$, which is the zero-th cohomology $H^0(GL(m)/B^+, K_{\la^+})$, is defined to be $Ind_{B^+}^{GL(m)}(K_{\la^+})$.
Since $det(C_{11})$ is invertible, each $H^0_{GL(m)}(\la^+)$ is isomorphic to 
a tensor product 
\[H^0_{GL(m)}(\la')\otimes det(C_{11})^{\la^+_m},\] where 
\[\la'=(\la^+_1-\la^+_m, \ldots, \la^+_{m-1}-\la^+_m,0)\]
is a weight with nonnegative entries. 

Assume now that $\la^+$ is dominant and all entries in $\la^+$ are nonnegative. 
We can realize $H^0_{GL(m)}(\la^+)$ as follows.
For indices $i_1, \ldots, i_s$ that are distinct elements of the set $\{1, \ldots, m\}$,
denote the determinant
\begin{equation*}D^+(i_1, \ldots, i_s)=
\begin{array}{|ccc|}
c_{1,i_1} & \ldots & c_{1,i_s} \\ 
c_{2,i_1} & \ldots & c_{2,i_s} \\ 
\ldots & \ldots & \ldots \\ 
c_{s,i_1} & \ldots & c_{s,i_s}%
\end{array}%
\end{equation*}
Denote by $D^+(s)$ any determinant $D^+(i_1, \ldots, i_s)$ of size $s$.
We say that any expression of type 
\[\prod_{a=1}^{m-1} D^+(a)^{\la^+_a-\la^+_{a+1}}D^+(m)^{\la^+_m}\]
is a bideterminant of shape $\la^+$.

Then any bideterminant of shape $\la^+$ is an element of $H^0_{GL(m)}(\la^+)$, and $H^0_{GL(m)}(\la^+)$ has a basis given by standard bideterminants of the shape $\la^+$ - 
see Section 4 of \cite{gr} or \cite{martin}.

The action of Lie algebra $\mathfrak{gl}(m)$ on the basis elements of $H^0_{GL(m)}(\la^+)$ is expressed in terms of the action of derivations 
$_{ij}D$, where $1\leq i,j\leq m$, as follows
\[(D^+(i_1,\ldots, i_s))_{ij}D= D^+(i_1, \ldots, \widehat{i_t}, j, \ldots, i_s) \] if $i=i_t$ for some $t=1,\ldots, s$ and $(D^+(i_1,\ldots, i_s))_{ij}D=0$ otherwise.

Analogous description applies to $GL(n)$, its weights $\la^-$ and induced $GL(n)$-modules $H^0_{GL(n)}(\la^-)$.

\subsection{$G_{ev}$-structure of induced $G$-supermodules}

Combining the above descriptions of induced $GL(m)$- and $GL(n)$-modules we can get a description of induced $G_{ev}$-modules.

The group $G_{ev}$ has a standard Borel subgroup $B_{ev}$ corresponding to lower triangular matrices $X$, where blocks $X_{12}$ and $X_{21}$ vanish.

Denote by $K_{\la}$ the one-dimensional $T$-supermodule of highest weight $\la$. We can consider $K_{\la}$ as a $B_{ev}$-supermodule via extending the action from $T$ to $B_{ev}$ trivially. 
The induced supermodule $H^0_{G_{ev}}(\la)$, which is the zero-th cohomology $H^0(G_{ev}/B_{ev}, K_{\la})$, is defined to be $Ind_{B_{ev}}^{G_{ev}}(K_{\la})$.

Analogously, the group $G$ has a standard Borel subgroup $B$ corresponding to lower triangular matrices. 
Let $\la$ be a weight of $G$ and let $K_{\la}$ be the one-dimensional $T$-module of highest weight $\la$. 
We can consider $K_{\la}$ as a $B$-module via extending the action from $T$ to $B$ trivially. 
The induced supermodule $H^0_{G}(\la)$, which is the zero-th cohomology $H^0(G/B, K_{\la})$, is defined to be $Ind_{B}^{G}(K_{\la})$.

For understanding of the $G_{ev}$-structure of the induced supermodule $H^0_G(\la)$, the following result that presents $H^0_G(\la)$ as a supermodule
embedded inside $K[G]$ is very important.

The $G$-supermodule $H^0_G(\lambda)$ is described explicitly using the
isomorphism $\tilde{\phi} :H^0_{G_{ev}}(\lambda)\otimes S(C_{12})\to H^0_G(\lambda)$ of superspaces
defined in Lemma 5.1 of \cite{z}. This map is a restriction of the multiplicative morphism $\phi:K[G]\to K[G]$ given on generators as
follows: 
\begin{equation*}
C_{11}\mapsto C_{11}, C_{21}\mapsto C_{21}, C_{12}\mapsto
C_{11}^{-1}C_{12}, C_{22}\mapsto C_{22}-C_{21}C_{11}^{-1}C_{12}.
\end{equation*}
The map $\tilde{\phi}$ is an isomorphism of $G_{ev}$-supermodules and its image is a $G$-subsupermodule of $K[G]$.
Using this map, we consider $H^0_{G_{ev}}(\la)$ embedded into $H^0_G(\la)$ and its highest vector $v$ is represented as a product of bideterminants 
of type
\begin{equation*}D^+(i_1, \ldots, i_s)=
\begin{array}{|ccc|}
c_{1,i_1} & \ldots & c_{1,i_s} \\ 
c_{2,i_1} & \ldots & c_{2,i_s} \\ 
\ldots & \ldots & \ldots \\ 
c_{s,i_1} & \ldots & c_{s,i_s}%
\end{array}%
\end{equation*}
and 
\begin{equation*}D^-(j_1, \ldots, j_t)=
\begin{array}{|ccc|}
\phi(c_{m+1,j_1}) & \ldots & \phi(c_{m+1,j_t}) \\ 
\phi(c_{m+2,j_1}) & \ldots & \phi(c_{m+2,j_t}) \\ 
\ldots & \ldots & \ldots \\ 
\phi(c_{m+t,j_1}) & \ldots & \phi(c_{m+s,j_t})%
\end{array},
\end{equation*}
where indices $i_1, \ldots, i_s$ are distinct elements of the set $\{1, \ldots, m\}$, and $j_1, \ldots, j_t$ 
are distinct elements of the set $\{m+1, \ldots, m+n\}$.
Namely, 
\[v=\prod_{a=1}^m D^+(1,\ldots, a)^{\la_a-\la_{a+1}}\prod_{b=1}^n D^-(m+1, \ldots, m+b)^{\la_{m+b} - \la_{m+b+1}}.\]

Denote by $D^+(s)$ any determinant $D^+(i_1, \ldots, i_s)$ of size $s$ and by $D^-(t)$ any determinant $D^-(j_1, \ldots, j_t)$ of size $t$.
Then any bideterminant of type \[\prod_{a=1}^m D^+(a)^{\la_a-\la_{a+1}}\prod_{b=1}^n D^-(b)^{\la_{m+b} - \la_{m+b+1}}\] is an element of $H^0_{G_{ev}}(\la)$. 
It is well-known that the induced $G_{ev}$-modules have a basis given by semistandard bideterminants of the shape $\la$ - for their description within our context consult \cite{fm}. 
The $G_{ev}$-module structure of $H^0_{G_{ev}}(\la)$ is completely described by action of even superderivations $_{ij}D$ on above bideterminants.

In order to describe the basis and the $G$-supermodule structure of $H^0_G(\la)$, action of odd superderivations $_{ij}D$ on bideterminants must be computed - this was done in Section 2 of \cite{fm}.
This description involves also products of elements $y_{kl}$ given as
\begin{equation*}
y_{kl}=\phi(c_{kl})=\frac{A_{k1}c_{1l}+A_{k2}c_{2l}+\ldots +A_{km}c_{ml}}{D}
\end{equation*}
for $1\leq k\leq m$ and $m+1\leq l \leq m+n$, 
where the matrix $A=(A_{ij})$ is the adjoint of the matrix $C_{11}$ and $D=Det(C_{11})$.

The span of all elements $y_{kl}$ for $1\leq k\leq m$ and $m+1\leq l\leq m+n$ is a $G_{ev}$-supermodule that will be denoted by $Y$.

\subsection{Largest polynomial subsupermodule $\nabla(\la)$ of $H^0_G(\la)$}

A $G$-supermodule $v$ is polynomial if and only if its coefficient space $cf(V)$ belongs to $A(m|n)$, 
that is if any $g\in G$ acts on $V$ in a such way that all matrix coefficients of $g$ (with respect to any basis of $V$) are polynomial functions in $g_{ij}$. 
It follows that all weights of a polynomial supermodule $V$ are such that all of their components are nonnegative.
If $\la$ is a highest weight of a polynomial $G$-supermodule, it is called the {\it polynomial} weight of $G$. Since $char(K)=0$, the polynomial weights of $G$ correspond to $(m|n)$-hook partitions - see \cite{br}. 
A complete description of polynomial weights if $char(K)=p\neq 2$ was obtained in \cite{bk}.
The largest polynomial subsupermodule of $H^0_G(\la)$ is denoted by $\nabla(\la)$. 

The category of polynomial $G$-supermodules of degree $r\geq 0$ is isomorphic to the category of supermodules over the Schur superalgebra $S(m|n,r)$.
Under this isomorphism, the supermodule $\nabla(\la)$ corresponds to the constandard module of the highest weight $\la$ for the Schur superalgebra $S(m|n,r)$. 
It has been proved in \cite{z} that the category of $G$-supermodules is a highest weight category. On the other hand, according to \cite{mz1} the category of supermodules over Schur superalgebra $S(m|n,r)$ is a highest weight category if and only if it is semisimple. Consequently, our understanding of the structure of $S(m|n,r)$ is much more elusive than that of 
$G$. For more information on $\nabla(\la)$ see Section 6 of \cite{z} and Section 5 of \cite{bk}.

In the classical case of a general linear group $GL(m)$, any induced module $H^0_{GL(m)}(\mu)$, where $\mu=(\mu_1, \ldots, \mu_m)$ is dominant, is isomorphic to a tensor multiple 
of the induced module $H^0_{GL(m)}(\mu_1-\mu_m, \ldots, \mu_{m-1}-\mu_m,0)$, which is polynomial $GL(m)$-module, and of the $\mu_m$-th power of the determinant for $GL(m)$. 
The structure of induced modules $H^0_G(\la)$ for polynomial $\mu$ is given by biterminants and is well understood. 
Therefore, the $GL(m)$-structure of all induced modules $H^0_{GL(m)}(\mu)$ is known.
By extension, the $G_{ev}$-structure of $H^0_{G_{ev}}(\la)$, for every dominant $\la$ as before, is determined and we will use it later.
 
The connection between the $G_{ev}$-module structure of $H^0_G(\la)$ and $\nabla(\la)$ is not so satisfactory as in the case of $GL(m)$-modules. 
The element $Ber(C) = det(C_{11}-C_{12}C_{22}^{-1}C_{21})det(C_{22})^{-1}$ generates a one-dimensional $G$-supermodule $Ber$ of the weight
$(1,\ldots, 1|-1,\ldots, -1)$. 
Since the Berezinian $Ber(C)$ is a group-like element, tensoring with powers of $Ber$ gives an isomorphism between corresponding induced supermodules. If $\la$ is polynomial, then the image of $\nabla(\la)$ is again a subsupermodule of the corresponding induced supermodule, although it need not be its polynomial subsupermodule. 
This way, by tensoring with powers of $Ber$ we can extend results derived for $\nabla(\la)$ to subsupermodules of certain, but not all, induced supermodules.

Assume now that a polynomial weight of $G$ corresponds to a $(m|n)$-hook partition $\la$.
For every $(m|n)$-hook partition $\la=(\la_1, \ldots, \ldots, \la_t)$ denote by $\la^+=(\la_1, \ldots, \la_{min\{m,t\}})$ and by $\la^-$ the transpose of the 
partition $(\la_{m+1}, \ldots, \la_{t})$. Note that $\la^-$ is nonzero only if $t>m$. The correspoding polynomial weight $\la$ of $G$ will be identified with $(\la^+|\la^-)$.

Denote by $S_{\mu}(x_1, \ldots, x_m)$ the Schur function and by $S_{\la/\mu}(y_1, \ldots, y_n)$ the skew Schur function.
It is well known (see \cite{br}) that 
the character of the induced $G$-supermodule of the highest weight $\la$ is given by the hook Schur fuction $HS(\la)$ described as:
\[HS_{\la}(x_1, \ldots, x_m;y_1, \ldots y_n)=\sum_{\mu<\la^+} S_{\mu}(x_1, \ldots, x_m) S_{\la'/ \mu'}(y_1, \ldots, y_n),\]
where ${\la'/ \mu'}$ is the conjugate of the skew partition $\la/\mu$. 
The hook Schur function can be also given as
$\sum_{T_{\la}(m,n) \,\, semistandard} T_{\la}(x_1, \ldots, x_m;y_1, \ldots y_n)$,
where $T_{\la}(x_1, \ldots, x_m;y_1, \ldots y_n)$ are polynomials counting the number of appearances of symbols $1$ through $m+n$ in tableaux $T$.

According to Theorem 6. 11 of \cite{br}, the dimension of even-primitive vectors of weight $(\mu|\nu)$ in $H^0_G(\la)$ is given as the Littlewood-Richardson coefficient 
$C^{\la'}_{\mu'\nu}$ from the decomposition of skew Schur function $S_{\la'/ \mu'}=\sum C^{\la'}_{\mu'\nu} S_{\nu}$ determined by the Littlewood-Richardson rule.
This multiplicity equals the number of Littlewood-Richardson tableux of shape ${\la'/ \mu'}$ and content $\nu$.
For more information on the above, consult \cite{br}, \cite{fulton} and \cite{van}.

\subsection{Primitive vectors $\pi_{I|J}$}

Primitive vectors play a crucial role in the representation theory of algebraic (groups and) supergroups because they generate subsupermodules within a given supermodule. Their description is important for the understanding of the structure of induced supermodules.

In order to define what $G$-primitive vector is, we need to consider a unipotent subsupergroup $U^-$ of $G$ corresponding to lower trangular matrices with all diagonal entries equal to $1$, and a unipotent subsupergroup $U^+$ of $G$ corresponding to upper trangular matrices with all diagonal entries equal to $1$. 
There is a Cartan decomposition $G=U^+TU^-$. A vector $v$ of weight $\la$ of a supermodule $M$ is called primitive vector if it is annihilated by any element of $U^-$.

Analogously, we have a Cartan decomposition $G_{ev}=U^+_{ev}TU^-_{ev}$, where $U^+_{ev}$ is the unipotent subgroup of $G_{ev}$ corresponding to lower triangular matrices from $G_{ev}$ with all diagonal entries equal to $1$, and $U^-_{ev}$ is the unipotent subgroup of $G_{ev}$ corresponding to lower triangular matrices from $G_{ev}$ with all diagonal entries equal to $1$. A vector $v$ of weight $\la$ of a module $M$ is called $G_{ev}$-primitive (or even-primitive vector)  if it is annihilated by any element of $U^-_{ev}$.

Since we are interested in even-primitive vectors inside of induced supermodule $H^0_G(\la)$ considered as a subsupermodule of $K[G]$, we can describe 
its even-primitive vectors using superderivations $_{ij}D$ as follows. A vector $v\in K[G]$ (or more generally $v\in K(m|n)$) is even-primitive vector if and only if 
$(v)_{ij}D=0$ whenever $ i>j$ and either $1\leq i,j \leq m$ or $m+1\leq i,j \leq m+n$.

Fix a dominant weight $\la$ of $G$.
Consider the multiindex \[(I|J)=(i_1\ldots i_k|j_1\ldots j_k)\] such that $1\leq i_1, \ldots, i_k \leq m$ and $1\leq j_1, \ldots, j_k\leq n$ and call $k$ the length of $I$ and $J$. 
Denote by $cont(I|J)=(cont(I)|cont(J))$ the content of $(I|J)$, the vector counting the number of appearences of symbols $1$ through $m+n$ in $(I|J)$.
Denote by $\delta^+_i$ the weight of $G$ with all components zeroes, except for the $i$-component that is equal to 1, and denote by 
$\delta^-_j$ the weight of $G$ with all components zeroes, except for the $m+j$-component that is equal to 1. Then define the weight $\la_{I|J}$ corresponding 
to $\la$ and multiindex $(I|J)$ by
\[\la_{I|J}=\la-\sum_{s=1}^k \delta^+_{i_s}+\sum_{s=1}^k \delta^-_{j_s}.\]

Assume $(I|J)=(i_1\ldots i_k|j_1\ldots j_k)$ is such that $\la_{I|J}$ is dominant.
If $i_1\leq i_2 \leq \ldots \leq i_k$ and $i_r=i_{r+1}$ implies $j_r<j_{r+1}$, then $(I|J)$ is called left admissible. 
Any $(I|J)$ that is left admissible will be called just admissible.
If $j_1\leq j_2 \leq \ldots \leq j_k$ and $j_s=j_{s+1}$ implies $i_s<i_{s+1}$, then $(I|J)$ is called right admissible.

For each $1\leq i \leq m$ and $1\leq j\leq n$ denote by $\rho_{i|j}$ the following element: 
\[\sum_{r=i}^m D^+(1, \ldots, i-1,r)\sum_{s=1}^{j} (-1)^{s+j} D^-(m+1, \ldots, \widehat{m+s},\ldots, m+j)y_{r,m+s}\]
of $A_{ev}(m|n)\otimes Y$. As customary, we define $D^-(\emptyset)=1$; this is used for $s=j=1$ when $D^-(m+1, \ldots, \widehat{m+s},\ldots, m+j)=1$.
For each $(I|J)=(i_1\ldots i_k|j_1\ldots j_k)$ as above
denote $\rho_{I|J}=\otimes_{s=1}^k \rho_{i_s|j_s}$. We will abuse the notation and consider $\rho_{I|J}$ as an element of $A_{ev}(m|n)\otimes Y^{\otimes k}$
via a map that sends $(f_1\otimes y_1)\otimes \ldots \otimes (f_k\otimes y_k)$ to $(f_1\ldots f_k)\otimes (y_1\otimes \ldots \otimes y_k)$, where $f_1\ldots f_k$ is the product in 
$A_{ev}(m|n)$.

Consider the following element:
\[v_{I|J}=\frac{v}{\prod_{s=1}^k D^+(1, \ldots, i_s)\prod_{s=1}^k D^-(m+1,\ldots, m+j_s-1)},\]
where, in particular, $D^-(m+1,\ldots, m+j_s-1)=1$ if $j_s=1$. 

Clearly the expression $v_{I|J}$ depends only on the content $cont(I|J)=(\iota|\kappa)$. By definition, $v_{I|J}$ belongs to $K(m|n)$. 

Define \[\pi_{I|J}=v_{I|J}\rho_{I|J}.\]

It follows from Proposition 3.4 of \cite{fm} using the argument in the proof of Lemma 4.1 of \cite{fm} that every $\pi_{I|J}$ is an even-primitive vector in $K(m|n)$.
Therefore any linear combination of vectors $\pi_{I|J}$, where content of $I|J$ is the same, is an even-primitive vector in $K(m|n)$. It is a crucial observation, which follows from the definition of $v_{I|J}$, that if such a linear combination belongs to $A(m|n)$ (which can be checked by verifying certain congruences modulo $D^+(1, \ldots, i_s)$ and $D^-(m+1,\ldots, m+j_s-1)$), then it also belongs to the supersubmodule $\nabla(\la)$ of $H^0_G(\la)$. This is the way we will construct even-primitive vectors of $\nabla(\la)$ and then of 
$H^0_G(\la)$.

The weight of $\pi_{I|J}$ that equals $(\la^+-\iota|\la^-+\kappa)$ will be denoted by $(\mu|\nu)$ and $v_{I|J}$ will be also denoted by 
$v_{\mu|\nu}$. The components $\iota$ and $\kappa$ of $cont(I|J)$ correspond to skew partitions $\la^+/ \mu$ and $\nu/ \la^-$, respectively. 

In \cite{fm}, the weight $\la$ is called $(I|J)$-robust if $v_{I|J}$ belongs to $A(m|n)$. It is clear that this happens if and only if 
the symbol $i_s<m$ appears at most $\la^+_{i_s}-\la^+_{i_s+1}$ times in $I$,
symbol $m$ appears at most $\la^+_m$ times in $I$; and 
symbol $j_t>1$ appears at most $\la^-_{j_t-1}-\la^-_{j_t}$ times in $J$.

From now on, the image of an element $x$ under the natural map $K(m|n) \otimes \otimes^k Y \to K(m|n)\otimes \wedge^k Y$ will be denoted by $\overline{x}$. In particular, 
$\overline{\rho}_{I|J}$ and $\overline{\pi}_{I|J}$, respectively, are images of $\rho_{I|J}$ and $\pi_{I|J}$, respectively.

Theorem 4.4 of \cite{fm} states that if $char(K)=0$, $(I|J)=(i_1\ldots i_k|j_1\ldots j_k)$ is admissible, $\la$ is $(I|J)$-robust, $\la_{I|J}=\tau$, and $\tau_m\geq n$, then the set of all vectors $\overline{\pi}_{K|L}$ for admissible $(K|L)$ such that $cont(K|L)=cont(I|J)$ form a basis of the set of even-primitive vectors of weight $\tau$ in $H^0_G(\la)$.

It is clear that each even-primitive vector in $H^0_G(\la)$ has a weight $\tau=\la_{I|J}$ for some admissible $(I|J)$. 
In this paper, in the case $char(K)=0$, we describe explicitly a basis of even-primitive vectors of weight $\tau$ in $H^0_G(\la)$ as linear combinations of vectors $\overline{\pi}_{K|L}$, 
where $cont(K|L)=cont(I|J)$. Each coefficient in this linear combination is either zero, or plus or minus one.

\section{Congruences modulo $D^+(1, \ldots, i)$ and $D^-(m+1, \ldots, m+j)$}

We start with the following lemma.

\begin{lm}\label{l3}
Let $1\leq i<m$ and $1\leq j_1, j_2 \leq n$. Then 
$\rho_{i,j_1}\otimes\rho_{i+1,j_2}-\rho_{i+1,j_1}\otimes \rho_{i,j_2}$ 
equals 
\[\begin{aligned}&\sum_{i\leq r_1<r_2\leq m}\sum_{s_1=1}^{j_1}\sum_{s_2=1}^{j_2} D^+(1,\ldots,i-1,i)D^+(1,\ldots,i-1,r_1,r_2)(-1)^{s_1+s_2+j_1+j_2}\times \\
&D^-(m+1,\ldots, \widehat{m+s_1}, \ldots, m+j_1)D^-(m+1, \ldots, \widehat{m+s_2}, \ldots, m+j_2)\times\\
&[y_{r_1,m+s_1} \otimes y_{r_2,m+s_2}-y_{r_2,m+s_1}\otimes y_{r_1,m+s_2}].
\end{aligned}\]
In particular, $\rho_{i,j_1}\otimes\rho_{i+1,j_2}-\rho_{i+1,j_1}\otimes \rho_{i,j_2} \equiv 0 \pmod {D^+(1, \ldots, i)}$.
\end{lm}
\begin{proof}
Write
\[\begin{aligned}&\rho_{i,j_1}\otimes\rho_{i+1,j_2}-\rho_{i+1,j_1}\otimes \rho_{i,j_2}=\\
&[\sum_{r_1=i}^m D^+(1, \ldots, i-1, r_1)\sum_{s_1=1}^{j_1}(-1)^{s_1+j_1}D^-(m+1, \ldots, \widehat{m+s_1}, \ldots, m+j_1) y_{r_1,m+s_1}]\otimes\\
&[\sum_{r_2=i+1}^m D^+(1, \ldots, i, r_2)\sum_{s_2=1}^{j_2}(-1)^{s_2+j_2}D^-(m+1, \ldots, \widehat{m+s_2}, \ldots, m+j_2) y_{r_2,m+s_2}]\\
-&[\sum_{r_2=i+1}^m D^+(1, \ldots, i, r_2)\sum_{s_1=1}^{j_1}(-1)^{s_1+j_1}D^-(m+1, \ldots, \widehat{m+s_1}, \ldots, m+j_1) y_{r_2,m+s_1}]\otimes\\
&[\sum_{r_1=i}^m D^+(1, \ldots, i-1, r_1)\sum_{s_2=1}^{j_2}(-1)^{s_2+j_2}D^-(m+1, \ldots, \widehat{m+s_2}, \ldots, m+j_2) y_{r_1,m+s_2}]\\
&=\sum_{r_1=i}^m\sum_{r_2=i+1}^m\sum_{s_1=1}^{j_1}\sum_{s_2=1}^{j_2} D^+(1,\ldots,i-1,r_1)D^+(1,\ldots,i,r_2)(-1)^{s_1+s_2+j_1+j_2}\times\\
&D^-(m+1,\ldots, \widehat{m+s_1}, \ldots, m+j_1)D^-(m+1, \ldots, \widehat{m+s_2}, \ldots, m+j_2)\times\\
&[y_{r_1,m+s_1}\otimes y_{r_2,m+s_2}-y_{r_2,m+s_1}\otimes y_{r_1,m+s_2}]\\
\end{aligned}\]
and break it up into two sums 
\begin{equation*}{(*)}
\begin{aligned}&\sum_{r_1=i+1}^m\sum_{r_2=i+1}^m\sum_{s_1=1}^{j_1}\sum_{s_2=1}^{j_2} D^+(1,\ldots,i-1,r_1)D^+(1,\ldots,i,r_2)(-1)^{s_1+s_2+j_1+j_2}\times\\
&D^-(m+1,\ldots, \widehat{m+s_1}, \ldots, m+j_1)D^-(m+1, \ldots, \widehat{m+s_2}, \ldots, m+j_2)\times\\
&[y_{r_1,m+s_1} \otimes y_{r_2,m+s_2}-y_{r_2,m+s_1}\otimes y_{r_1,m+s_2}] 
\end{aligned}
\end{equation*}
and
\[\begin{aligned}
&\sum_{r_2=i+1}^m\sum_{s_1=1}^{j_1}\sum_{s_2=1}^{j_2} D^+(1,\ldots,i-1,i)D^+(1,\ldots,i,r_2)(-1)^{s_1+s_2+j_1+j_2}\times\\
&D^-(m+1,\ldots, \widehat{m+s_1}, \ldots, m+j_1)D^-(m+1, \ldots, \widehat{m+s_2}, \ldots, m+j_2)\times\\
&[y_{i,m+s_1} \otimes y_{r_2,m+s_2}-y_{r_2,m+s_1}\otimes y_{i,m+s_2}]
\end{aligned}.\]

If $r_1=r_2$, then the corresponding contribution in the first sum equals zero.
Consider now $i+1\leq r_1\neq r_2\leq m$.

The determinantal identity 
\[\begin{aligned}&D^+(1, \ldots, i-1,r_1)D^+(1, \ldots, i,r_2)-D^+(1, \ldots, i-1, r_2)D^+(1, \ldots, i,r_1)\\
&=D^+(1, \ldots, i-1,r_1,r_2)D^+(1, \ldots, i-1,i)\end{aligned}\]
follows from the identity
\[c_{i,r_1}\,\begin{array}{|cc|}c_{i,i}&c_{i,r_2}\\c_{i+1,i}&c_{i+1,r_2}\end{array}-
c_{i,r_2}\,\begin{array}{|cc|}c_{i,i}&c_{i,r_1}\\c_{i+1,i}&c_{i+1,r_1}\end{array}=
\begin{array}{|cc|}c_{i,r_1}&c_{i,r_2}\\c_{i+1,r_1}&c_{i+1,r_2}\end{array}\,c_{i,i}\]
using the law of extensible minors (see \cite{bs} and \cite{fm}), by adding new rows $1,\ldots, i-1$ and new columns $1, \ldots, i-1$.

Therefore
\[\begin{aligned}&D^+(1, \ldots, i-1,r_1)D^+(1, \ldots, i,r_2)[y_{r_1,m+s_1}\otimes y_{r_2,m+s_2}-y_{r_2,m+s_1}\otimes y_{r_1,m+s_2}]\\
&+D^+(1, \ldots, i-1,r_2)D^+(1, \ldots, i,r_1)[y_{r_2,m+s_1} \otimes y_{r_1,m+s_2}- y_{r_1,m+s_1}\otimes y_{r_2,m+s_2}]\\
&=[D^+(1, \ldots, i-1,r_1)D^+(1, \ldots, i,r_2)-D^+(1, \ldots, i-1,r_2)D^+(1, \ldots, i,r_1)]\times \\
&[y_{r_1,m+s_1}\otimes y_{r_2,m+s_2}-y_{r_2,m+s_1}\otimes y_{r_1,m+s_2}]\\
&=D^+(1, \ldots, i-1,r_1,r_2)D^+(1, \ldots, i-1,i)[y_{r_1,m+s_1}\otimes y_{r_2,m+s_2}-y_{r_2,m+s_1}\otimes y_{r_1,m+s_2}]
\end{aligned}.\]

Using this, we can rewrite the expression $(*)$ as
\[\begin{aligned}&\sum_{i<r_1<r_2\leq m}\sum_{s_1=1}^{j_1}\sum_{s_2=1}^{j_2} D^+(1,\ldots,i-1,i)D^+(1,\ldots,i-1,r_1,r_2)(-1)^{s_1+s_2+j_1+j_2}\times\\
&D^-(m+1,\ldots, \widehat{m+s_1}, \ldots, m+j_1)D^-(m+1, \ldots, \widehat{m+s_2}, \ldots, m+j_2)\times\\
&[y_{r_1,m+s_1} \otimes y_{r_2,m+s_2}-y_{r_2,m+s_1}\otimes y_{r_1,m+s_2}]
\end{aligned} 
\]
and $\rho_{i,j_1}\otimes\rho_{i+1,j_2}-\rho_{i+1,j_1}\otimes \rho_{i,j_2}$ as
\[\begin{aligned}&\sum_{i\leq r_1<r_2\leq m}\sum_{s_1=1}^{j_1}\sum_{s_2=1}^{j_2} D^+(1,\ldots,i-1,i)D^+(1,\ldots,i-1,r_1,r_2)(-1)^{s_1+s_2+j_1+j_2}\times\\
&D^-(m+1,\ldots, \widehat{m+s_1}, \ldots, m+j_1)D^-(m+1, \ldots, \widehat{m+s_2}, \ldots, m+j_2)\times\\
&[y_{r_1,m+s_1} \otimes y_{r_2,m+s_2}-y_{r_2,m+s_1}\otimes y_{r_1,m+s_2}].
\end{aligned}\]
\end{proof}

\begin{lm}\label{l4}
Let $1\leq i_1,i_2\leq m$ and $1\leq j < n$. Then $\rho_{i_1,j}\otimes \rho_{i_2,j+1}-\rho_{i_1,j+1}\otimes \rho_{i_2,j}$ equals
\[\begin{aligned}&\sum_{r_1=i_1}^m\sum_{r_2=i_2}^m\sum_{1\leq s_1<s_2\leq j+1} D^+(1, \ldots, i_1-1,r_1)D^+(1, \ldots, i_2-1,r_2)(-1)^{s_1+s_2+1}\times \\
&D^-(m+1, \ldots, m+j)D^-(m+1, \ldots, \widehat{m+s_1}, \ldots, \widehat{m+s_2}, \ldots, m+j+1)\times \\
&[y_{r_1,m+s_1}\otimes y_{r_2,m+s_2}-y_{r_1,m+s_2}\otimes y_{r_2,m+s_1}].
\end{aligned}\]
In particular, $\rho_{i_1,j}\otimes \rho_{i_2,j+1}-\rho_{i_1,j+1}\otimes \rho_{i_2,j} \equiv 0 \pmod{D^-(m+1, \ldots, m+j)}$.
\end{lm}
\begin{proof}
The proof of this lemma is similar to the proof of the Lemma \ref{l3} and we will only provide its shorter outline.
Write $\rho_{i_1,j}\otimes \rho_{i_2,j+1}-\rho_{i_1,j+1}\otimes \rho_{i_2,j}$ as a sum of 
\begin{equation*}(**)\begin{aligned}&\sum_{r_1=i_1}^m\sum_{r_2=i_2}^m\sum_{s_1=1}^j\sum_{s_2=1}^j D^+(1, \ldots, i_1-1,r_1)D^+(1, \ldots, i_2-1,r_2)(-1)^{s_1+s_2+1}\times \\
&D^-(m+1, \ldots, \widehat{m+s_1},\ldots, m+j)D^-(m+1, \ldots, \widehat{m+s_2}, \ldots, m+j+1)\times \\
&[y_{r_1,m+s_1}\otimes y_{r_2,m+s_2}-y_{r_1,m+s_2}\otimes y_{r_2,m+s_1}] 
\end{aligned}\end{equation*}
and 
\[\begin{aligned}&\sum_{r_1=i_1}^m\sum_{r_2=i_2}^m\sum_{s_1=1}^j D^+(1, \ldots, i_1-1,r_1)D^+(1, \ldots, i_2-1,r_2)(-1)^{s_1+j}\times \\
&D^-(m+1, \ldots, \widehat{m+s_1},\ldots, m+j)D^-(m+1, \ldots, \ldots, m+j)\times \\
&[y_{r_1,m+s_1}\otimes y_{r_2,m+j+1}-y_{r_1,m+j+1}\otimes y_{r_2,m+s_1}]. 
\end{aligned}\]

If $s_1=s_2$, then the corresponding contribution in the first sum equals zero.
Consider now $1\leq s_1\neq s_2\leq j$.

The determinantal identity 
\[\begin{aligned}&D^-(m+1, \ldots, \widehat{m+s_1},\ldots, m+j)D^-(m+1, \ldots, \widehat{m+s_2}, \ldots, m+j+1)\\
&-D^-(m+1, \ldots, \widehat{m+s_2},\ldots, m+j)D^-(m+1, \ldots, \widehat{m+s_1}, \ldots, m+j+1)\\
&=D^-(m+1, \ldots, m+j)D^-(m+1, \ldots, \widehat{m+s_1}, \ldots, \widehat{m+s_2}, \ldots, m+j+1)\end{aligned}\]
follows from the identity
\[\begin{aligned}&c_{m+s_1,m+s_2}\,\begin{array}{|cc|}c_{m+s_1,m+s_1}&c_{m+s_1,m+j+1}\\c_{m+s_2,m+s_1}&c_{m+s_2,m+j+1}\end{array}-
c_{m+s_1,m+s_1}\,\begin{array}{|cc|}c_{m+s_1,m+s_2}&c_{m+s_1,m+j+1}\\c_{m+s_2,m+s_2}&c_{m+s_2,m+j+1}\end{array}\\
=&\,\begin{array}{|cc|}c_{m+s_1,m+s_1}&c_{m+s_1,m+s_2}\\c_{m+s_2,m+s_1}&c_{m+s_2,m+s_2}\end{array}\,c_{m+s_1,m+j+1}
\end{aligned}\]
using the law of extensible minors (see \cite{bs} and \cite{fm}), by adding rows 
\[m+1,\ldots, \widehat{m+s_1}, \ldots, \widehat{m+s_2}, \ldots, m+j\] and columns 
\[m+1,\ldots, \widehat{m+s_1}, \ldots, \widehat{m+s_2}, \ldots, m+j.\]

Therefore
\[\begin{aligned}&D^-(m+1, \ldots, \widehat{m+s_1},\ldots, m+j)D^-(m+1, \ldots, \widehat{m+s_2}, \ldots, m+j+1)\times\\
&[y_{r_1,m+s_1}\otimes y_{r_2,m+s_2}-y_{r_1,m+s_2}\otimes y_{r_2,m+s_1}]\\
&+D^-(m+1, \ldots, \widehat{m+s_2},\ldots, m+j)D^-(m+1, \ldots, \widehat{m+s_1}, \ldots, m+j+1)\times\\
&[y_{r_1,m+s_2}\otimes y_{r_2,m+s_1}-y_{r_1,m+s_1}\otimes y_{r_2,m+s_2}]\\
&=D^-(m+1, \ldots, m+j)D^-(m+1, \ldots, \widehat{m+s_1}, \ldots, \widehat{m+s_2}, \ldots, m+j+1)\times\\
&[y_{r_1,m+s_1}\otimes y_{r_2,m+s_2}-y_{r_1,m+s_2}\otimes y_{r_2,m+s_1}].
\end{aligned}\]
Using this, we rewrite $\rho_{i_1,j}\otimes \rho_{i_2,j+1}-\rho_{i_1,j+1}\otimes \rho_{i_2,j}$ in the stated form.
\end{proof}

\section{Determinantal identities}
In the previous section we have used determinantal identities to derive congruences modulo $D^+(1, \ldots, i)$ and $D^-(m+1, \ldots, m+j)$. 
Before we can get a description of $G_{ev}$-primitive vectors in $H^0_{G_{ev}}(\la)\otimes \otimes^k Y$, we need to derive additional determinantal identities.

\begin{lm}\label{d2}
Let $2\leq s \leq m$, $x_1, \ldots, x_{s-1}$ and $a_1, \ldots, a_s$ be integers from the set $\{1, \ldots, m\}$. Then there is the following determinantal identity
\[\begin{aligned}&\sum_{t=1}^s (-1)^{s-t} D^+(x_1, \ldots, x_{s-1}, a_t)D^+(a_1, \ldots, \widehat{a_t}, \ldots, a_s)\\
&=D^+(x_1, \ldots, x_{s-1})D^+(a_1, \ldots, a_s).
\end{aligned}\]
\end{lm}
\begin{proof}
For each $t$, use the Laplace expansion by the last column to express
\[D^+(x_1, \ldots, x_{s-1}, a_t)=\sum_{u=1}^{s} (-1)^{s-u} c_{ua_t}M_u,\] where $M_u$ is the minor corresponding to the position $(u,s)$. Since each $M_u$ does not depend on $t$, 
we can rewrite the left-hand-side of the above identity as
\[\sum_{u=1}^{s} (-1)^{s-u} M_u [\sum_{t=1}^s (-1)^{s-t} c_{ua_t} D^+(a_1, \ldots, \widehat{a_t}, \ldots, a_s)].\]

The expression $\sum_{t=1}^s (-1)^{s-t} c_{ua_t} D^+(a_1, \ldots, \widehat{a_t}, \ldots, a_s)$ is the Laplace expansion by the last row of the determinant  
$\begin{array}{|ccc|}
c_{1,a_1} & \ldots & c_{1,a_s} \\ 
c_{2,a_1} & \ldots & c_{2,a_s} \\ 
\ldots & \ldots & \ldots \\ 
c_{s-1,a_1} & \ldots & c_{s-1,a_s}\\
c_{u,a_1} & \ldots & c_{u,a_s}
\end{array}$. 
This vanishes if $u\neq s$ and equals $D^+(a_1, \ldots, a_s)$ if $u=s$. Since $M_s=D^+(x_1, \ldots, x_{s-1})$ the claim follows.
\end{proof}

Note that the last identity applied to $s=2$ and $x_1=1$ 
\[D^+(1, a_2) D^+(a_1) -D^+(1, a_1)D^+(a_2) =D^+(1)D^+(a_1,a_2)\]
is essentially the relationship we have used earlier (after relabeling and applying the law of extensible minors).

Combining multiple identities from Lemma \ref{d2} we obtain the following statement. 

\begin{lm}\label{d3}
Let $2\leq s \leq m$, $x_1, \ldots, x_{s-1}$ and $a_1, \ldots, a_s$ be integers from the set $\{1, \ldots, m\}$. Let $\sigma$ denote permutations of $(a_1, \ldots, a_s)$.
Then 
\[\sum_{\sigma} (-1)^{\sigma} \prod_{t=1}^s D^+(x_1, \ldots, x_{t-1}, \sigma(a_t))=D^+(a_1, \ldots, a_s)\prod_{t=1}^{s-1} D^+(x_1, \ldots, x_t).\]
\end{lm} 
\begin{proof}
We proceed by induction on $s$. The base case $s=2$ is settled in Lemma \ref{d2}. Assuming the formulas are valid for $s=u$, consider $s=u+1$. 
Fix $(b_1, \ldots, b_{u+1})$ with entries in the set $\{1, \ldots, m\}$ and consider permutations $\tau$ of $(b_1, \ldots, b_{u+1})$. 
There is a unique cyclic permutation $\gamma$ such that $\gamma(b_{u+1})=\tau(b_{u+1})$ and $\tau\circ \gamma^{-1}$ can be identified with its restriction $\sigma$ on 
the set $\{b_1, \ldots, b_{u+1}\}/ \tau(b_{u+1})$ which is viewed as $(a_1, \ldots, a_u)=(\gamma(b_1), \ldots, \gamma(b_u))$. Thus $\tau=\sigma\gamma$.

Then 
\[\begin{aligned}&\sum_{\tau} (-1)^{\tau} \prod_{t=1}^{u+1} D^+(x_1, \ldots, x_{t-1}, \tau(b_t)) = \\
&\sum_{\gamma} [\sum_{\sigma} (-1)^{\sigma} \prod_{t=1}^u D^+(x_1, \ldots, x_{t-1}, \sigma(a_t))] (-1)^{\gamma} D^+(x_1, \ldots, x_u, \gamma(b_{u+1}))= \\
&\prod_{t=1}^{u-1} D^+(x_1, \ldots, x_t)\sum_{\gamma} (-1)^{\gamma} D^+(x_1, \ldots, x_u, \gamma(b_{u+1}))D^+(a_1, \ldots, a_u))
\end{aligned}\]
by the inductive assumption, which equals 
\[\prod_{t=1}^{u-1} D^+(x_1, \ldots, x_t)D^+(x_1, \ldots, x_u)D^+(b_1, \ldots, b_u, b_{u+1})\]
by Lemma \ref{d2}.
\end{proof}
The following proposition will be applied in the proof of Theorem \ref{p1} to the case when $(a_1, \ldots, a_s)$ are entries in a row of a skew tableau $T^+$, where the first entry is positioned in its 
$u$-th column (and the last entry in its $(u+s-1)$-st column).

\begin{pr}\label{pr1}
Let $1\leq u $ and $2\leq s $ be such that $u+s-1\leq m$ and $a_1, \ldots, a_s$ be integers from the set $\{1, \ldots, m\}$. Let $\sigma$ denote permutations of $(a_1, \ldots, a_s)$.
Then 
\[\begin{aligned}
&\sum_{\sigma} (-1)^{\sigma} \prod_{t=1}^s D^+(1, \ldots, u+t-2, \sigma(a_t))\\
&=D^+(1, \ldots, u-1,a_1, \ldots, a_s)\prod_{t=1}^{s-1} D^+(1, \ldots, u+t-1),
\end{aligned}\]
where $\sigma$ runs through all permutations of $(a_1, \ldots, a_s)$.
\end{pr}
\begin{proof}
Setting $x_{1}=u, \ldots, x_{s-1}=u+s-2$ in Lemma \ref{d3} yields
\[\sum_{\sigma} (-1)^{\sigma} \prod_{t=1}^s D^+(u, \ldots, u+t-2, \sigma(a_t))=D^+(a_1, \ldots, a_s)\prod_{t=1}^{s-1} D^+(u, \ldots, u+t-1)\]
and the statement follows using the law of extensible minors by inserting indices $1, \ldots, u-1$. 
\end{proof}

\begin{rem}\label{rem1}
If there is an index $t$ such that $\sigma(a_t)\leq u+t-2$, then the product $\prod_{t=1}^s D^+(1, \ldots, u+t-2, \sigma(a_t))$ vanishes. Therefore in the sum 
\[\sum_{\sigma} (-1)^{\sigma} \prod_{t=1}^s D^+(1, \ldots, u+t-2, \sigma(a_t))\] of the above proposition we could consider only permutations $\sigma$ that satisfy
$u+t-1\leq \sigma(a_t)$ for each index $t$.
\end{rem}

The next lemma is an analogue of Lemma \ref{d2}.

\begin{lm}\label{d2'}
Let $2\leq s \leq n$, $x_1, \ldots, x_{s-1}$ and $a_1, \ldots, a_s$ be integers from the set $\{m+1, \ldots, m+n\}$. Then there is the following determinantal identity
\[\begin{aligned}&\sum_{t=1}^s (-1)^{s-t} D^-(x_1, \ldots, x_{s-1}, a_t)D^-(a_1, \ldots, \widehat{a_t}, \ldots, a_s)\\
&=D^-(x_1, \ldots, x_{s-1})D^-(a_1, \ldots, a_s).
\end{aligned}\]
\end{lm}

Assume $1\leq u$ and $2\leq s$ are such that $u+s-1\leq n$.
Fix $(a_1, \ldots, a_s)$ such that $m+1\leq a_1 < \ldots <a_s\leq m+n$ 
and denote by $\mathcal{O}=\mathcal{O}_u(a_1, \ldots, a_s)$ the set of those permutations 
$\sigma$ of $(a_1, \ldots, a_s)$  that satisfy $\sigma(a_t)\leq m+u+t-1$ for each $t=1, \ldots, s$.

Our next goal is to prove the following proposition.

\begin{pr}\label{pr2}
Let $1\leq u $ and $2\leq s$ be such that $u+s-1\leq n$ and $m+1\leq a_1 < \ldots <a_s\leq m+u+s-1$.
%  are such that $a_t\leq m+u+t-1$ for each $t$. 
Then 
\[\begin{aligned}&\sum_{\sigma\in \mathcal{O}} (-1)^{\sigma} \prod_{t=1}^s  D^-(m+1, \ldots, \widehat{\sigma(a_t)}, \dots, m+u+t-1)\\
=&D^-(m+1, \ldots, \widehat{a_1}, \ldots, \widehat{a_t}, \ldots, \widehat{a_s}, \ldots, m+u+s-1)\times \\
&\prod_{t=1}^{s-1} D^-(m+1, \ldots, m+u+t-1).
\end{aligned}\]
\end{pr}
\begin{proof}
For simplicity, rewrite 
$D^-(m+1, \ldots, \widehat{a_1}, \ldots, \widehat{a_t}, \ldots, \widehat{a_s}, \ldots, m+u+s-1)$ as $D^-(b_1, \ldots, b_{u-1})$, 
where $m+1\leq b_1 <\ldots <b_{u-1}\leq m+u+s-1$.

In the first step, using Lemma \ref{d2'} we rewrite 
\[\begin{aligned}&D^-(b_1, \ldots, b_{u-1})D^-(m+1, \ldots, m+u)\\
&=\sum_{d_1=1}^{m+u} (-1)^{m+u-d_1} D^-(b_1, \ldots, b_{u-1}, d_1) D^-(m+1, \ldots, \widehat{d_1}, m+u).
\end{aligned}\]
Proceeding by induction, in the $t$-th step we rewrite
\[\begin{aligned}&D^-(b_1, \ldots, b_{u-1}, d_1, \ldots, d_{t-1})D^-(m+1, \ldots, m+u+t-1)\\
&=\sum_{d_t=1}^{m+u} (-1)^{m+u+t-1-d_t} D^-(b_1, \ldots, b_{u-1}, d_1, \ldots, d_t) D^-(m+1, \ldots, \widehat{d_t}, m+u+t-1).
\end{aligned}\]
Combining the above expresssions, after $s-1$ steps we obtain
\[\begin{aligned}
&D^-(b_1, \ldots, b_{u-1})\prod_{t=1}^{s-1} D^-(m+1, \ldots, m+u+t-1)\\
&=\sum_{d_1=1}^{m+u}\ldots \sum_{d_{s-1}=1}^{m+u+s-2} (-1)^{(m+u)+\ldots + (m+u+s-2) -d_1-\ldots d_{s-1}}\times \\
&D^-(b_1, \ldots, b_{u-1}, d_1, \ldots, d_{s-1})\prod_{t=1}^{s-1} D^-(m+1, \ldots, \widehat{d_t}, m+u+t-1).
\end{aligned}\]
Applying Lemma \ref{d2'} one more time we obtain 
\[\begin{aligned}&D^-(b_1, \ldots, b_{u-1}, d_1, \ldots, d_{s-1})D^-(m+1, \ldots, m+u+s-1)=(-1)^{m+u+s-1-d_s}\times \\
&D^-(b_1, \ldots, b_{u-1}, d_1, \ldots, d_{s-1}, d_s)D^-(m+1, \ldots, \widehat{d_s}, \ldots, m+u+s-1)\\
=&(-1)^{m+u+s-1-d_s}\epsilon D^-(m+1, \ldots, m+u+s-1) D^-(m+1, \ldots, \widehat{d_s}, \ldots, m+u+s) 
\end{aligned}\]
for a unique $d_s$. Here
\[D^-(b_1, \ldots, b_{u-1}, d_1, \ldots, d_{s-1})=(-1)^{m+u+s-1-d_s}\epsilon D^-(m+1, \ldots, \widehat{d_s}, \ldots, m+u+s),\] 
where $\epsilon=\pm 1$ is determined by 
\[D^-(b_1, \ldots, b_{u-1}, d_1, \ldots, d_{s-1}, d_s)=\epsilon D^-(m+1, \ldots, m+u+s-1).\]
Additionally, the only nonzero summands in the above sum for \[D^-(b_1, \ldots, b_{u-1})\prod_{t=1}^{s-1} D^-(m+1, \ldots, m+u+t-1)\] are 
given by those $(d_1, \ldots, d_s)$ that are permutations of $(a_1, \ldots, a_s)$ satisfying $m+1 \leq d_t\leq m+s+t-1$ for each $t=1, \ldots, s$.
Hence $(d_1, \ldots, d_s)=\sigma(a_1, \ldots, a_s)$ for $\sigma\in \mathcal{O}$.

Thus
\[\begin{aligned}
&D^-(b_1, \ldots, b_{u-1})\prod_{t=1}^{s-1} D^-(m+1, \ldots, m+u+t-1)\\
&=\sum_{d_1=1}^{m+u}\ldots \sum_{d_{s}=1}^{m+u+s-1} (-1)^{(m+u)+\ldots + (m+u+s-1) -d_1-\ldots d_s}\epsilon\times \\
&\qquad \qquad \qquad \qquad \prod_{t=1}^{s} D^-(m+1, \ldots, \widehat{d_t}, m+u+t-1).
\end{aligned}\]

To determine the sign $\epsilon$, consider $D^-(m+1, \ldots, m+u+s-1)$, shift the symbols at positions $a_s+1$ through $m+u+s-1$ one place to the left and move the symbol $a_s$ to the position $m+u+s-1$. 
This is accomplished by applying $m+u+s-1-a_s$ transpositions. Then proceed by induction from $t=s$ to $t=1$, in the $t$-th step shift the entries at positions $a_t+1$ 
through $m+u+t-1$ one place to the left and move the symbol $a_{s-1}$ to the position $m+u+t-1$, by applying $m+u-t-1-a_t$ transpositions. 
After $s$ steps we arrive at $D^-(b_1, \ldots, b_u,a_1, \ldots, a_s)$. This shows that 
\[\begin{aligned}&D^-(b_1, \ldots, b_u,d_1, \ldots, d_s)=(-1)^{\sigma} D^-(b_1, \ldots, b_u,a_1, \ldots, a_s)\\
&= (-1)^{(m+u)+\ldots + (m+u+s-1) -d_1-\ldots d_s} (-1)^{\sigma} D^-(m+1, \ldots, m+u+s-1).
\end{aligned}\]

Consequently, 
\[\begin{aligned}
&D^-(b_1, \ldots, b_{u-1})\prod_{t=1}^{s-1} D^-(m+1, \ldots, m+u+t-1)\\
&=\sum_{\sigma\in \mathcal{O}} (-1)^{\sigma} \prod_{t=1}^{s} D^-(m+1, \ldots, \widehat{\sigma(a_t)}, m+u+t-1).
\end{aligned}\]
\end{proof}

\begin{rem}\label{rem2}
We define the expression $D^-(m+1, \ldots, \widehat{a}, \ldots, m+j)$ such that $m+j<a$ to be equal to zero. Therefore in Proposition \ref{pr2} we can
replace the summation over the set $\mathcal{O}$ by the summation over the set of all permutations $\sigma$ of $(a_1, \ldots, a_s)$.
\end{rem} 

\section{Operators $\sigma^+$, $\sigma^-$ and $\sigma$ on $\rho_{I|J}$}

From now on, assume that $\la=(\la^+|\la^-)$ is a $(m|n)$-hook partition, unless stated otherwise. 

In order to combine various previously defined operators $\sigma^+_i$ and $\sigma^-_j$, we need to introduce the terminology of skew diagrams and tableaux. 

\subsection{General setup for diagrams and tableaux}\label{gensetup}

Let $\alpha$ and $\beta$ be partitions. We define the partial order $<$ on partitions by requiring that $\beta<\alpha$ if and only if $\beta_i\leq \alpha_i$ for every $i$. 

Let $\beta<\alpha$ and $\mathcal{D}$ be the diagram corresponding to the skew partition $\alpha/ \beta$. The canonical column skew tableau $D^+_{can}$ of the shape $\alpha/ \beta$ is defined in such a way that its $j$-th column is filled with entries $j$ for each $j$. Analogously, the canonical row skew tableau $D^-_{can}$ of the shape $\alpha/ \beta$ is defined so that its $i$-th row is filled with entries $m+i$ for each $i$. 

Let $\la=(\la^+|\la^-)$ be a $(m|n)$-hook partition, and a partition $\mu$ is such that $\mu<\la^+$.
Let $[\la]$ be the diagram corresponding to $\la=(\la^+|\la^-)$ and $[\la'/ \mu']$ be a skew diagram corresponding to the skew partition $\la'/ \mu'$. 
We will denote by $T$ a skew tableau of shape $\la'/ \mu'$ that is filled with (possibly repeated) entries from the set $m+1, \ldots, m+n$ such that its content 
equals $(0|\nu)$, where $\nu$ is a partition. Additionally, we will assume that, for each $j>m$ and $1\leq i\leq n$, the entry at the position $[ij]$ of the tableau $T$ 
is equal to $m+i$. These assumptions imply that $\la^-<\nu$. 
Such tableau $T$ consists of two parts. The first part $T^+$ is a tableau of the shape $(\la^+/ \mu)'$ and content $(0|\nu / \omega)$.
The second part is the canonical row tableau $L^-_{can}$ coresponding to the diagram $[\la^-]$. Therefore $\omega=\la^-$ as partitions.

We will denote by $T^{opp}$ a tableau of the shape $\nu$ and content $(\la^+ /\mu|\la^-)$. 
Additionally, we will assume that for $1\leq i\leq n$ and each $1\leq j\leq \la^-_i$ the entry at the position $[ij]$ of the tableau $T^{opp}$ is equal to $m+i$. 
Such tableau $T^{opp}$ consists of two parts. The first part is the canonical row tableau $L^-_{can}$ corresponding to the diagram $[\omega]$. The second part is a skew tableau $T^-$ of the shape $\nu/ \omega$ and content
$(\la^+ /\mu|0)$. Although $\omega=\la^-$ as partitions, we distinguish them to differentiate between parts of tableaux $T$ and $T^{opp}$.

Denote by $\mathcal{D}^+$ the diagram $[\la^{+'}/ \mu']$ coresponding to $T^+$ and by $\mathcal{D}^-$ the diagram $[\nu/ \omega]$ corresponding to $T^-$.

\begin{ex}
To illustrate the above notation, consider $m=n=3$,
$\la=(6,4,4,3,1)$, $\la^+=(6,4,4)$, $\la^-=(3,1)$, $(\la^-)'=(2,1,1)$, $\mu=(4,2,1)$ and $\nu=(5,4,2)$. Then the tableau
\[\young(aaaabb,aabb,abbb,ccc,c)\]
is such that the entries $a$ represent the shape $\mu$; entries $b$ represent the shape $\la^+/\mu$; and entries $c$ represent the shape $(\la^-)'$.

Transpose of this tableau is 
\[\young(aaacc,aabc,abbc,abb,b,b)\]
and its entries $a$ represent the shape $\mu'$; entries $b$ represent the shape $(\la^+)'/\mu'$; and the entries $c$ represent the shape $\la^-$. 

The tableau 
\[\young(ccbbb,cbbb,cb)\]
is such that the entries $c$ represent the shape $\la^-=\omega$, and the entries $b$ represent the shape $\nu/\omega$.

One example of the tableau $T$ is
\[\young(:::44,::55,:446,:65,4,5)\]
and the corresponding tableau $T^+$ is 
\[\young(::5,:44,:65,4,5)\]
An example of the tableau $T^{opp}$ is
\[\young(44321,5331,62)\]
and the corresponding tableau $T^-$ is 
\[\young(::321,:331,:2).\]
\end{ex}

\subsection{Definition of operators and positioning maps}\label{tensorsetup}

In order to define operator $\sigma^+$ on $\rho_{I|J}$, we need to consider a positioning map
$P^{+}$ and tableaux $T^+_{can}$ and $T^{+}$ of shape $(\la^+/ \mu)'$. Fix a multiindex $(I|J)$ of length $k$. 
Choose a positioning map $P^{+}:\{1, \ldots, k\} \to \mathcal{D}^+$ that is a bijection and satisfies the property that $P^+(l)=[s,r]$ implies $i_l=r$.
The tableau $T^+_{can}$ is the canonical tableau defined by the property that for each $1\leq i\leq m$ its $i$-th column consists of entries equal to $i$.
The tableau $T^{+}$ is given by $J$ and $P^+$ in such a way that its entry at the position $P^+(l)=[s,r]$ equals $j_l$ (and $r=i_l$).
Let $X^+$ be the subgroup of the symmetric group $\Sigma_k$ consisting of row permutations of $\mathcal{D}^+$. 
For $\sigma\in X^+$ denote by $\sigma(T^+_{can})$ the tableau obtained by applying permutation $\sigma$ to the entries of $T^+_{can}$.

The action of $\sigma\in X^+$ on $\rho_{I|J}$ is given as $\sigma.\rho_{I|J}=\rho_{K|J}$, where for each $1\leq l\leq k$ the index $k_l$ is the entry at the position $P^{+}(l)$ 
 in $\sigma(T^+_{can})$. Finally, the operator $\sigma^+$ is given as 
\[\sigma^+=\sum_{\sigma\in X^+} (-1)^\sigma \sigma.\]

Analogously, in order to define operator $\sigma^-$ on $\rho_{I|J}$ we need to consider a positioning map
$P^{-}$ and tableaux $T^-_{can}$ and $T^{-}$ of shape $\nu/ \omega$. Choose a positioning map 
$P^{-}:\{1, \ldots, k\} \to \mathcal{D}^-$ that is a bijection and satisfies the property that $P^-(l)=[r,s]$ implies $j_l=m+r$.
The tableau $T^-_{can}$ is the canonical tableau that is defined by the property that for each $1\leq j\leq n$ its $j$-th row consists of entries equal to $m+j$.
The tableau $T^{-}$ is given by $I$ and $P^-$ in such a way that its entry at the position $P^-(l)=[r,s]$ equals $i_l$ (and $m+r=j_l$).
Let $X^-$ be the subgroup of the symmetric group $\Sigma_k$ consisting of column permutations of $\mathcal{D}^-$. For
$\sigma\in X^-$ denote by $\sigma(T^-_{can})$ the tableau obtained by applying permutation $\sigma$ to the entries of $T^-_{can}$.

The action of  $\sigma\in X^-$ on $\rho_{I|J}$ is given as $\sigma.\rho_{I|J}=\rho_{I|L}$, where for each $1\leq a\leq k$ the index $l_a$ is the entry at the position 
$P^{-}(a)$ in $\sigma(T^-_{can})$.
Finally, the operator $\sigma^-$ is given as 
\[\sigma^-=\sum_{\sigma\in X^-} (-1)^\sigma \sigma.\]

Note the difference between definitions of $\sigma^+$ and $\sigma^-$. It is due to the presence of the transpose in the shape of $T^+$. Therefore the canonical tableau $T^+_{can}$
is filled by columns and $\sigma^+$ involves row permutations, while the canonical tableau $T^-_{can}$ is filled by rows and $\sigma^-$ involves column permutations.

\begin{ex}
If $T^+_{can}=\young(:2,12,12)$, then 
$\sigma^+T^+_{can}=\young(:2,12,12)-\young(:2,21,12)-\young(:2,12,21)+\young(:2,21,21)$.
If $T^-_{can}=\young(333,44)$, then $\sigma^-T^-_{can}=\young(333,44)-\young(433,34)-\young(343,43)+\young(443,33).$
\end{ex}

Denote by $S^+_s=\{[s, p_s], \ldots, [s,p_s+\ell^+_s-1]\}$ the set of entries in the $s$-th row of $\mathcal{D}^+$, by 
$SP^+_s$ the set $(P^+)^{-1}(S^+_s)$ and by $\widehat{SP^+_s}$ the complement of $SP^+_s$ in $\{1, \ldots, k\}$. 
List elements of $SP^+_s$ in the fixed order $\{q_1, \ldots, q_{\ell^+_s}\}$, where $q_t=(P^+)^{-1}([s, p_s+t-1])$. Then $i_{q_t}=p_s+t-1$ for each $1\leq t \leq \ell^+_s$.

Denote by $E^+_s$ the embedding of $\otimes_{q\in SP^+_s} Y$ to $\otimes^k Y$ such that the component $Y$ corresponding to $q\in SP^+_s$ is mapped identically to the 
$t$-th component of $\otimes^k Y$ while other components of $\otimes^k Y$ corresponding to $q\in \widehat{SP^+_s}$ are all equal to $1$. Then the map
$\otimes_s E_s$ is an isomorphism between $\otimes_s (\otimes_{q\in SP^+_s} Y)$ and $\otimes^k Y$.
Using this isomorphism we identify $\rho_{I|J}$ with $\otimes_s \rho^+_s$, where $\rho^+_s=\otimes_{q\in SP^+_s} \rho_{i_t,j_t}$ (each $i_t=s$ here).
Denote by $X^+_s$ a subset of $X^+$ consisting of those $\sigma_s$ that only permute entries in the $s$-th row of $\mathcal{D}^+$. Then the action of $\sigma_s$ on 
$\otimes^k Y$ induces the corresponding action on $\otimes_{q\in SP^+_s} Y$ and the action of $\sigma_s$ on $\rho_{I|J}$ restricts to the action on $\rho^+_s$.
Morever, every $\sigma\in X^+$ can be written as a product $\sigma=\prod_s \sigma_s$ and the action of $\sigma$ on $\otimes^k Y$ and $\rho_{I|J}$ breaks down to the 
products of the commuting actions of $\sigma_s$. Therefore, it is enough to analyze actions of $\sigma_s\in X^+_s$ separately.  

Let ${R}=(R_1, \ldots, R_{\ell^+_s})$ be an arbitrary $\ell^+_s$-tuple of integers such that $i_{q_t}=p_s+t-1\leq R_t$ for each $1\leq t \leq \ell^+_s$. 
If $\sigma_s\in SP^+_s$, then $\sigma_s.(R_1, \ldots, R_{\ell^+_s})$ is the $\ell^+_s$-tuple obtained by permuting entries of $R$ by $\sigma_s$.
We call $R$ ordered provided $R_1<R_2<\ldots <R_{\ell^+_s}$. Denote by $\mathcal{O}(R_1, \ldots, R_{\ell^+_s})$ 
the set of all ${\ell^+_s}$-tuples $(r_1, \ldots, r_{\ell^+_s})$ that are permutations of $(R_1, \ldots, R_{\ell^+_s})$. 

Denote \[DI^+(R_1, \ldots, R_{\ell^+_s})=D^+(1, \ldots, i_{q_1}-1, R_1)\ldots D^+(1, \ldots, i_{q_{\ell^+_s}}-1, R_{\ell^+_s})\]
and define the action of $\sigma_s\in SP^+_s$ on $DI^+(R_1, \ldots, R_{\ell^+_s})$ by
\[\sigma_s.DI^+(R_1, \ldots, R_{\ell^+_s})=DI^+(\sigma_s.(R_1, \ldots, R_{\ell^+_s})).\]

Finally, denote $\sigma^+_s=\sum_{\sigma_s\in X^+_s} (-1)^{\sigma_s} \sigma_s$.

Analogously, denote by $S^-_r=\{[p_r, r], \ldots, [p_r+\ell^-_r-1,r]\}$ the set of entries in the $r$-th column of $\mathcal{D}^-$, by 
$SP^-_r$ the set $(P^-)^{-1}(S^-_r)$ and by $\widehat{SP^-_r}$ the complement of $SP^-_r$ in $\{1, \ldots, k\}$. 
List elements of $SP^-_r$ in the fixed order $\{q_1, \ldots, q_{\ell^-_r}\}$, where $q_t=(P^-)^{-1}([p_r+t-1,r])$. Then $j_{q_t}=p_r+t-1$ for each $1\leq t \leq \ell^-_r$.

Denote by $E^-_r$ the embedding of $\otimes_{q\in SP^-_r} Y$ to $\otimes^k Y$ such that the component $Y$ corresponding to $q\in SP^-_r$ is mapped identically to the 
$t$-th component of $\otimes^k Y$ while other components of $\otimes^k Y$ corresponding to $q\in \widehat{SP^-_r}$ are all equal to $1$. Then the map
$\otimes_r E_r$ is an isomorphism between $\otimes_s (\otimes_{q\in SP^-_r} Y)$ and $\otimes^k Y$.
Using this isomorphism we identify $\rho_{I|J}$ with $\otimes_r \rho^-_r$, where $\rho^-_r=\otimes_{q\in SP^-_r} \rho_{i_t,j_t}$ (each $j_t=r$ here).
Denote by $X^-_r$ a subset of $X^-$ consisting of those $\sigma_r$ that only permute entries in the $r$-th column of $\mathcal{D}^-$. Then the action of $\sigma_r$ on 
$\otimes^k Y$ induces the corresponding action on $\otimes_{q\in SP^-_r} Y$ and the action of $\sigma_r$ on $\rho_{I|J}$ restricts to the action on $\rho^-_r$.
Morever, every $\sigma\in X^-$ can be written as a product $\sigma=\prod_r \sigma_r$ and the action of $\sigma$ on $\otimes^k Y$ and $\rho_{I|J}$ breaks down to the 
products of the commuting actions of $\sigma_r$. Therefore, it is enough to analyze actions of $\sigma_r\in X^-_r$ separately.  

Let ${S}=(S_1, \ldots, S_{\ell^-_r})$ be an arbitrary $\ell^-_r$-tuple of integers such that $S_t\leq j_{q_t}=p_r+t-1$ for each $1\leq t \leq \ell^-_r$. 
If $\sigma_r\in SP^-_r$, then $\sigma_r.(S_1, \ldots, S_{\ell^-_r})$ is the $\ell^-_r$-tuple obtained by permuting entries of $S$ by $\sigma_r$.
We call $S$ ordered provided $S_1<S_2<\ldots <S_{\ell^-_r}$. Denote by $\mathcal{O}(S_1, \ldots, S_{\ell^-_r})$ 
the set of all ${\ell^-_r}$-tuples $(s_1, \ldots, s_{\ell^-_r})$ that are permutations of $(S_1, \ldots, S_{\ell^-_r})$. 

Denote \[\begin{aligned}DJ^-(S_1, \ldots, S_{\ell^-_r})=&(-1)^{S_1+\ldots+ S_{\ell^-_r}+j_{q_1}+\ldots j_{q_{\ell^-_r}}}D^-(m+1, \ldots, \widehat{m+S_1}, \ldots, m+j_{q_1}) \\
&\ldots  D^-(m+1, \ldots, \widehat{m+S_{\ell^-_r}}, \ldots, m+j_{q_{\ell^-_r}})
\end{aligned}\]
and define the action of $\sigma_r\in SP^-_r$ on $DJ^-(S_1, \ldots, S_{\ell^-_r})$ by
\[\sigma_r.DJ^-(S_1, \ldots, S_{\ell^-_r})=DJ^-(\sigma_r.(S_1, \ldots, S_{\ell^-_r})).\]

Finally, denote $\sigma^-_rs=\sum_{\sigma_r\in X^-_r} (-1)^{\sigma_r} \sigma_r$.
and define the operator $\sigma_{\pm}$ as a composition of $\sigma^-\circ \sigma^+$. 
In what follows we will write $\sigma$ for $\sigma_{\pm}$. From the context we should be able to distinguish this operator $\sigma$ from elements
$\sigma\in X^+, X^-$.

\subsection{Even-primitive vectors in $H^0_{G_{ev}}(\la)\otimes \otimes^k Y$} 

\begin{teo}\label{p1}
Let $\la$ be a $(m|n)$-hook partition and $(I|J)$ be admissible multiindex of length $k$. Then, for any choice of positioning maps $P^{+}$ and $P^{-}$,
the vector $v_{I|J}\sigma.\rho_{I|J}$ 
(which is an integral linear combination of $\pi_{K|L}$, where $cont(K|L)=cont(I|J)$) is a $G_{ev}$-primitive vector in $H^0_{G_{ev}}(\la)\otimes \otimes^k Y$.
\end{teo}
\begin{proof}
Write 
\[\begin{aligned} v_{I|J}\sigma .\rho_{I|J}= &\frac{\sigma .\rho_{I|J}}{\prod_{i=1}^{m-1} D^+(1, \ldots, i)^{\la_{i+1}-\mu_i}\prod_{j=1}^{n-1} D^-(m+1,\ldots, m+j)^{\nu_{j+1}-\la^-_{j}}}\times \\&D^+(1, \ldots, m)^{\mu_m}D^-(m+1,\ldots, m+n)^{\la^-_{n}}.
\end{aligned}\]

We will show that $\sigma .\rho_{I|J}$ is a multiple of 
\[\prod_{i=1}^{m-1} D^+(1, \ldots, i)^{\la_{i+1}-\mu_i}\prod_{j=1}^{n-1} D^-(m+1,\ldots, m+j)^{\nu_{j+1}-\la^-_{j}}\]
and therefore $v_{I|J}\sigma .\rho_{I|J}$ is a linear combination of products of (polynomial) determinants and elements from $\bigotimes^k Y$.
Moreover, we will see that the products of determinants, appearing in this description, are bideterminants of the type
\[\prod_{a=1}^m D^+(a)^{\la_a-\la_{a+1}}\prod_{b=1}^n D^-(b)^{\la_{m+b} - \la_{m+b+1}}, \] and are therefore elements of $H^0_{G_{ev}}(\la)$.
We infer from there that $v_{I|J}\sigma .\rho_{I|J}\in H^0_{G_{ev}}(\la)\otimes \otimes^k Y$.

We will use the determinantal identities of Propositions \ref{pr1} and \ref{pr2} to replace determinants. These identities keep the sizes and multiplicities of determinants and therefore keep the shape of the bideterminant that is the product of such determinants. For example, to illustrate this, using the identity of type 
$D^+(i)^{\la_i-\mu_i}=D^+(1, \ldots, i)^{\la_{i+1}-\mu_i}D^+(i)^{\la_i-\la_{i+1}}$ we could rewrite $\frac{D^+(i)^{\la_i-\mu_i}}{D^+(1, \ldots, i)^{\la_{i+1}-\mu_i}}$ 
as $D^+(i)^{\la_i-\la_{i+1}}$.

Using the definition of $\rho$ we express 
\[\begin{aligned}&\rho^+_s=\sum_{r_1=i_{q_1}}^m \ldots \sum_{r_{\ell^+_s}=i_{q_{\ell^+_s}}}^m 
\sum_{s_1=1}^{j_{q_1}} \ldots \sum_{s_{\ell^+_s}=1}^{j_{q_{\ell^+_s}}}\\
& D^+(1, \ldots, i_{q_1}-1, r_1)\ldots D^+(1, \ldots, i_{q_{\ell^+_s}-1}, r_{\ell^+_s})\times (-1)^{s_1+\ldots+ s_{\ell^+_s}+j_{q_1}+\ldots j_{q_{\ell^+_s}}}\times \\
&D^-(m+1, \ldots, \widehat{m+s_1}, \ldots, m+j_{q_1})\ldots  D^-(m+1, \ldots, \widehat{m+s_{\ell^+_s}}, \ldots, m+j_{q_{\ell^+_s}})\times \\
&E^+_s(y_{r_1,m+s_1}\otimes \ldots \otimes y_{r_{\ell^+_s},m+s_{\ell^+_s}}),
\end{aligned}\]

and rewrite
\[\begin{aligned}&\rho^+_s= \sum_{s_1=1}^{j_{q_1}} \ldots \sum_{s_{\ell^+_s}=1}^{j_{q_{\ell^+_s}}}
(-1)^{s_1+\ldots+ s_{\ell^+_s}+j_{q_1}+\ldots j_{q_{\ell^+_j}}}\times \\
&D^-(m+1, \ldots, \widehat{m+s_1}, \ldots, m+j_{q_1})\ldots  D^-(m+1, \ldots, \widehat{m+s_{\ell^+_s}}, \ldots, m+j_{q_{\ell^+_s}})\\
&\big[\sum_{r_1=i_{q_1}}^m \ldots \sum_{r_{\ell^+_s}=i_{q_{\ell^+_s}}}^m 
DI^+(r_1, \ldots r_{\ell^+_s})E^+_s(y_{r_1,m+s_1}\otimes \ldots \otimes y_{r_{\ell^+_s},m+s_{\ell^+_s}})\Big].
\end{aligned}\]

Therefore

\[\begin{aligned}&\sigma^+_s.\rho^+_s= \sum_{s_1=1}^{j_{q_1}} \ldots \sum_{s_{\ell^+_s}=1}^{j_{q_{\ell^+_s}}}
(-1)^{s_1+\ldots+ s_{\ell^+_s}+j_{q_1}+\ldots j_{q_{\ell^+_j}}}\times \\
&D^-(m+1, \ldots, \widehat{m+s_1}, \ldots, m+j_{q_1})\ldots  D^-(m+1, \ldots, \widehat{m+s_{\ell^+_s}}, \ldots, m+j_{q_{\ell^+_s}})\\
&\big[\sum_{r_1=i_{q_1}}^m \ldots \sum_{r_{\ell^+_s}=i_{q_{\ell^+_s}}}^m 
DI^+(r_1, \ldots r_{\ell^+_s})\sigma^+_s.E^+_s(y_{r_1,m+s_1}\otimes \ldots \otimes y_{r_{\ell^+_s},m+s_{\ell^+_s}})\Big].
\end{aligned}\]

Fix indices $s_1, \ldots, s_{\ell^+_s}$, and $R=(R_1, \ldots, R_{\ell^+_s})$ such that $i_{q_t}=p_s+t-1\leq R_t$ for each $1\leq t \leq \ell^+_s$ and consider 
$(r_1, \ldots, r_{\ell^+_s})\in \mathcal{O}(R_1, \ldots, R_{\ell^+_s})$.

The sum 
\[\sum_{(r_1, \ldots, r_{\ell^+_s})\in \mathcal{O}(R_1, \ldots, R_{\ell^+_s})}DI^+(r_1,\ldots, r_{\ell^+_s})
\sigma^+_s.E^+_s(y_{r_1,m+s_1}\otimes \ldots \otimes y_{r_{\ell^+_s},m+s_{\ell^+_s}})\] can be
rearranged as 
\[\begin{aligned}&\Big[\sum_{\sigma_s\in X^+_s} (-1)^{\sigma_s}
\sigma_s.DI^+(R_1,\ldots, R_{\ell^+_s})\Big]\times\\
&\Big[\sum_{\sigma_s\in X^+_s} (-1)^{\sigma_s}
\sigma_s.E^+_i(y_{R_1,m+s_1}\otimes \ldots \otimes y_{R_{\ell^+_s},m+s_{\ell^+_s}})\Big]\\
&=[\sigma^+_s. DI^+(R_1,\ldots, R_{\ell^+_s})][\sigma^+_s. E^+_i(y_{R_1,m+s_1}\otimes \ldots \otimes y_{R_{\ell^+_s},m+s_{\ell^+_s}})]
\end{aligned}\]
because \[\sigma^+_s.E^+_i(y_{r_1,m+s_1}\otimes \ldots \otimes y_{r_{\ell^+_s},m+s_{\ell^+_s}})=
\sigma^+_s.E^+_i(y_{R_1,m+s_1}\otimes \ldots \otimes y_{R_{\ell^+_s},m+s_{\ell^+_s}})\]
for each $(r_1, \ldots, r_{\ell^+_s})\in \mathcal{O}(R_1, \ldots, R_{\ell^+_s})$.
Since 
\[DI^+(R_1,\ldots, R_{\ell^+_s})= \prod_{t=1}^{\ell^+_s} D^+(1,\ldots, p_s+t-2, R_t),\]
we get  
\[\sigma_s.DI^+(R_1,\ldots, R_{\ell^+_s})=\prod_{t=1}^{\ell^+_s} D^+(1,\ldots, p_s+t-2, \sigma(R_t))\]
and 
\[\sigma^+_s.DI^+(R_1,\ldots, R_{\ell^+_s})=\sum_{\sigma_s\in X^+_s} \prod_{t=1}^{\ell^+_s} D^+(1,\ldots, p_s+t-2, \sigma(R_t))\]
which by Proposition \ref{pr1} equals
\[D^+(1, \ldots, p_s-1,R_1, \ldots, R_{\ell^+_s})\prod_{t=1}^{\ell^+_s-1} D^+(1, \ldots, p_s+t-1).\]

Note that the last expression vanishes if entries of $R$ are not pairwise different. 

This implies that the term 
\[\sum_{r_1=i_{q_1}}^m \ldots \sum_{r_{\ell^+_s}=i_{q_{\ell^+_s}}}^m 
DI^+(r_1, \ldots r_{\ell^+_s})\sigma^+_s.E^+_s(y_{r_1,m+s_1}\otimes \ldots \otimes y_{r_{\ell^+_s},m+s_{\ell^+_s}}),\]
appearing in the expression for $\sigma^+_s \rho^+_s$,
equals the sum of the expressions
\[\begin{aligned}&\sigma^+_s.DI^+(R_1, \ldots R_{\ell^+_s}) \sigma^+_s.E^+_s(y_{R_1,m+s_1}\otimes \ldots \otimes y_{R_{\ell^+_s},m+s_{\ell^+_s}})\\
&=D^+(1, \ldots, p_s-1,R_1, \ldots, R_{\ell^+_s})\prod_{t=1}^{\ell^+_s-1} D^+(1, \ldots, p_s+t-1) \times \\
&\sigma^+_s.E^+_s(y_{R_1,m+s_1}\otimes \ldots \otimes y_{R_{\ell^+_s},m+s_{\ell^+_s}})
\end{aligned}\]
over all ordered $R=(R_1, \ldots R_{\ell^+_s})$ such that   $i_{q_t}=p_s+t-1\leq R_t$ for each $1\leq t \leq \ell^+_s$.

Combining contributions from all $\sigma^+_s$ we derive that 
$\sigma^+\rho_{I|J}$ is a multiple of 
$\prod_{i=1}^{m-1} D^+(1, \ldots, i)^{\la_{i+1}-\mu_i}$.

Analogously, we write $\sigma^-_r\rho^-_r$ as
\[\begin{aligned}&\sigma^-_r.\rho^-_r= \sum_{r_1=i_{q_1}}^m \ldots \sum_{r_{\ell^-_r}=i_{q_{\ell^-_r}}}^m 
D^+(1, \ldots, i_{q_1}-1, r_1)\ldots D^+(1, \ldots, i_{q_{\ell^-_r}-1}, r_{\ell^-_r})\\
&\big[\sum_{s_1=1}^{j_{q_1}} \ldots \sum_{s_{\ell^-_r}=1}^{j_{q_{\ell^-_r}}}
DJ^-(s_1, \ldots s_{\ell^-_r})\sigma^-_r.E^-_r(y_{r_1,m+s_1}\otimes \ldots \otimes y_{r_{\ell^-_r},m+s_{\ell^-_r}})\Big], 
\end{aligned}.\]
We can use Proposition \ref{pr2} to rewrite the expression
\[\sum_{s_1=1}^{j_{q_1}} \ldots \sum_{s_{\ell^-_r}=1}^{j_{q_{\ell^-_r}}}
DJ^-(s_1, \ldots s_{\ell^-_r})\sigma^-_r.E^-_r(y_{r_1,m+s_1}\otimes \ldots \otimes y_{r_{\ell^-_r},m+s_{\ell^-_r}})\]
as the sum of the expressions
\[\begin{aligned}&\sigma^-_r.DJ^-(S_1, \ldots S_{\ell^-_r}) \sigma^-_r.E^-_r(y_{r_1,m+S_1}\otimes \ldots \otimes y_{r_{\ell^-_r},m+S_{\ell^-_r}})\\
&=D^-(m+1, \ldots, m+p_s-1,m+S_1, \ldots, m+S_{\ell^-_r})\times \\
&\prod_{t=1}^{\ell^-_r-1} D^-(m+1, \ldots, m+p_s+t-1) \times \sigma^-_r.E^-_r(y_{r_1,m+S_1}\otimes \ldots \otimes y_{r_{\ell^-_r},m+S_{\ell^-_r}})
\end{aligned}\]
over all ordered $S=(S_1, \ldots S_{\ell^-_r})$ such that $S_t\leq j_{q_t}=p_r+t-1$ for each $1\leq t \leq \ell^-_r$.

Combining contributions from all $\sigma^-_r$ we derive that 
$\sigma^-\rho_{I|J}$ is a multiple of 
$\prod_{j=1}^{n-1} D^-(m+1, \ldots, m+j)^{\nu_{j+1}-\la^-_j}$.

Finally, we can combine actions $\sigma^+$ and $\sigma^-$ into $\sigma$. Since the maps $E^+_s$ and $E^-_r$ commute (because they act on $I$- and $J$- components of $\rho_{I|J}$, respectively), and determinantal identities involving $DI^+(r_1, \ldots, r_{\ell^+_s})$ and $DJ^-(s_1, \ldots, s_{\ell^-_r})$ are independent of each other (they involve determinants
of type $D^+(i_1, \ldots, i_t)$ and $D^-(j_1, \ldots, j_t)$, respectively),  
we can combine arguments used for $\sigma^+$ and $\sigma^-$ to conlude that 
$\sigma\rho_{I|J}$ is a multiple of 
\[\prod_{i=1}^{m-1} D^+(1, \ldots, i)^{\la_{i+1}-\mu_i}\prod_{j=1}^{n-1} D^-(m+1,\ldots, m+j)^{\nu_{j+1}-\la^-_{j}}\]
which cancels off all the terms in the denominator of $v_{I|J}$.
During this process, we group together those expressions that correspond to simultaneous row permutations of a tableau of the shape $\la'/ \mu'$ with entries strictly increasing in its rows from left to right (corresponding to $(R_1, \ldots R_{\ell^+_s})$ above) and column permutations of a tableau of the shape $\nu/ \omega$ with entries strictly increasing in its columns from top to bottom (corresponding to $(S_1, \ldots S_{\ell^-_r})$ above).

Or, alternatively, we can observe that the action of $\sigma^+$ and $\sigma^-$ commute since they act on the $I$- and $J$- component respectively. 
Since $\sigma^+(\sigma^-\rho_{I|J})$
is a linear combination of expressions of type $\sigma^+\rho_{I|L}$, we derive that $\sigma\rho_{I|J}$ is a multiple of $\prod_{i=1}^{m-1} D^+(1, \ldots, i)^{\la_{i+1}-\mu_i}$.
Since $\sigma^-(\sigma^+\rho_{I|J})$ is a linear combination of expressions of type $\sigma^-\rho_{K|J}$, we derive that $\sigma\rho_{I|J}$ is a multiple of 
$D^-(m+1,\ldots, m+j)^{\nu_{j+1}-\la^-_{j}}$. Since the variables in $D^+(1, \ldots, i)$ and $D^-(m+1,\ldots, m+j)$ are distinct, we conclude that 
$\sigma\rho_{I|J}$ is a multiple of  \[\prod_{i=1}^{m-1} D^+(1, \ldots, i)^{\la_{i+1}-\mu_i}\prod_{j=1}^{n-1} D^-(m+1,\ldots, m+j)^{\nu_{j+1}-\la^-_{j}}.\]

Moreover, the above formulas also show that $\sigma.\pi_{I|J}=v_{I|J}\sigma\rho_{I|J}$ is a linear combination of tensor products of bideterminants of the shape $\la$ and elements 
from $\otimes^k Y$, which means that it belongs to $H^0_{G_{ev}}(\la)\otimes \otimes^k Y$.
\end{proof}

The above theorem provides $G_{ev}$-primitive vectors of all possible weights in $H^0_{G_{ev}}(\la)\otimes \otimes^k Y$. It seems plausible that these vectors span
all $G_{ev}$-primitive vectors in $H^0_{G_{ev}}(\la)\otimes \otimes^k Y$, but we do not need this result and will not pursue it further in this paper.

\begin{cor}\label{korak}
The image $v_{I|J}\sigma.\overline{\rho}_{I|J}$ of $v_{I|J}\sigma.\rho_{I|J}$ in $H^0_{G_{ev}}(\la)\otimes \wedge^k Y$ is a $G_{ev}$-primitive vector in $H^0_G(\la)$.
\end{cor}

For the understanding of the $G_{ev}$-structure of $H^0_G(\la)=H^0_{G_{ev}}\otimes \wedge(Y)$, we need to find a basis of $G_{ev}$-primitive vectors in $H^0_G(\la)$. 
Since, in the simplest case when $H^0_G(\la)$ is irreducible, the dimensions of these vectors of a given weight are given by a number of certain Littlewood-Richardson tableaux, it is clear that finding such a basis is not a trivial problem. In the second half of this paper, we will search for an explicit basis of $G_{ev}$-primitive vectors of a given weight in $H^0_G(\la)$ and $\nabla(\la)$, consisting of the vectors of the above type $v_{I|J}\sigma.\overline{\rho}_{I|J}$. 

\section{Operators on tableaux and even-primitive vectors in $H^0_G(\la)$}

While Theorem \ref{p1} gives even-primitive vectors in $H^0_{G_{ev}}(\la)\otimes \otimes^k Y$, Corollary \ref{korak} gives even-primitive vectors in $H^0_G(\la)$.
From now on, we will concentrate our attention to a more detailed description of the even-primitive vectors in $H^0_G(\la) = H^0_{G_{ev}}(\la)\otimes \wedge(Y)$.
We will follow the general setup from Subsection \ref{gensetup}.

The starting point of our construction of the operator $\sigma^+$ was the congruence 
\[\rho_{i,j_1}\otimes\rho_{i+1,j_2}-\rho_{i+1,j_1}\otimes\rho_{i,j_2} \equiv 0 \pmod{D^+(1, \ldots, i)}\]
that lead to the definition of a positioning map $P^+$ and the operator $\sigma^+$, which is permuting entries in the $I$-component of $\rho_{I|J}$ while keeping the $J$-component unchanged.  

Since 
\[\begin{aligned}&\rho_{i,j_1}\wedge\rho_{i+1,j_2}-\rho_{i+1,j_1}\wedge\rho_{i,j_2}= \rho_{i,j_1}\wedge\rho_{i+1,j_2}+\rho_{i,j_2}\wedge\rho_{i+1,j_1}\\
&\equiv 0 \!\!\!\!\!\pmod{D^+(1, \ldots, i)}\end{aligned}\]
by Lemma \ref{l3}, when working over the exterior algebra, we can adjust the definition of the operator $\sigma^+$ (new operator will be denoted by $\tau^+$) in such a way that the $I$-component of $\rho_{I|J}$ is unchanged and the entries in the $J$-component are permuted. Also, in this case we have to remove the negative sign in the corresponding sum.

Analogously, the starting point of our construction of the operator $\sigma^-$ was the congruence 
\[\rho_{i_1,j}\otimes \rho_{i_2,j+1}-\rho_{i_1,j+1}\otimes \rho_{i_2,j} \equiv 0 \pmod{D^-(m+1, \ldots, m+j)}.\]
that lead to the definition of a positioning map $P^-$ and the operator $\sigma^-$, which is permuting entries in the $J$-component of $\rho_{I|J}$ while keeping the $I$-component unchanged.  

Since 
\[\begin{aligned}&\rho_{i_1,j}\wedge\rho_{i_2,j+1}-\rho_{i_1,j+1}\wedge\rho_{i_2,j}= \rho_{i_1,j}\wedge\rho_{i_2,j+1}+\rho_{i_2,j}\wedge\rho_{i_1,j+1} \\
&\equiv 0 \!\!\!\!\!\pmod{D^-(m+1, \ldots, m+j)}\end{aligned}\]
by Lemma \ref{l4}, when working over the exterior algebra, we can adjust the definition of the induced operator $\sigma^-$ (new operator will be denoted by $\tau^-$) in such a way that the $J$-component of $\rho_{I|J}$ is unchanged and the entries in the $I$-component are permuted. Also, in this case we have to remove the negative sign in the corresponding sum.

\subsection{Operators $\tau^+$ and $\tau^-$}

In Subsection \ref{tensorsetup}, using a positioning map $P^+$, to a multiindex $(K|L)$ we have assigned a tableaux $T^+$ of shape $(\la^+/ \mu)'$ and content $(0|\nu/ \omega)$ corresponding to $(K|L)$.

Conversely, to a tableau $T^+$ of shape $(\la^+/ \mu)'$ and content $(0|\nu/ \omega)$ we assign the multiindex $(I|J)$, in the following way.
The entries in $I$ are symbols $1\leq i\leq m$, they are weakly increasing, and for each $i$ 
there are exactly $\la^+_i-\mu_i$ entries in $I$ that are equal to $i$. 
The entries in $J$ are symbols $1\leq j\leq n$ and they are obtained by subtracting $m$ from 
entries in $T^+$ listed by columns from left to right, in each column ordered from top to bottom. 
Denote by $Q^+$ the map $T^+ \mapsto (I|J)$ defined this way.
In the particular case when the entries in rows of $T^+$ are strictly increasing from left to right and $\la_{I|J}$ is dominant, we obtain that $(I|J)$ is left admissible.
 
If $Q^+(T^+)=(I|J)$, then the entries in $I$ are weakly increasing, hence we will not obtain all possible multiindexes $(K|L)$ as images under $Q^+$. 
However, since we are working inside
the exterior algebra $\wedge(Y)$ instead of the tensor algebra over $Y$, after reordering of the terms in $K$ we get $\overline{\rho}_{K|L}=\epsilon\overline{\rho}_{I|J}$, where 
$\epsilon = \pm 1$ and $(I|J)=Q^+(T^+)$ for some $T^+$.
Additionally, define the vector $\overline{\rho}(T^+)=\overline{\rho}_{I|J}$.

If $T'^+$ is obtained from $T^+$ by column permutations, then $\overline{\rho}(T'^+)$ $=\pm \overline{\rho}(T^+)$ $=\pm \overline{\rho}_{I|J}$. Using the
correspondence between $T^+$ and $(I|J)$, we can replace the operator $\sigma^+$ acting on multiindices $(I|J)$ by an operator $\tau^+$ acting on tableaux $T^+$.

For $\sigma\in X^+$ write $\sigma\overline{\rho}_{I|J}=\overline{\sigma\rho_{I|J}}$ and denote by 
$\sigma T^+$ the tableau obtained by applying permutation $\sigma$ to the entries of $T^+$. 
The tableau $\sigma T^+$ corresponds to a multiindex $(I|L)$, where $L$ has the same content as $J$. 
It is clear that 
$\overline{\rho}(\sigma T^+)=\sigma\overline{\rho}(T^+)$.

The operator $\tau^+$ acting on $T^+$ is defined as \[\tau^+ T^+=\sum_{\sigma\in X^+} \sigma T^+.\]

The map $\overline{\rho}$ can be extended to linear combinations of tableaux in a natural way. 
Then the operator $\overline{\rho}\tau^+$ applied to $T^+$ is given as 
\[\overline{\rho}(\tau^+ T^+)=\sum_{\sigma\in X^+} \overline{\rho}(\sigma T^+).\]

The expressions $\overline{\rho}\tau^+ T^+$ can be considered as row bipermanents corresponding to the pair of tableaux $T^+_{can}$ and $T^+$ based on formal symbols
 $\rho_{ij}$, see \cite{grs}.

The operators $\tau^+$ and $\sigma^+$ are compatible in the sense that 
\begin{equation}\label{eq1}
\overline{\rho}(\tau^+ T^+)=\sigma^+\overline{\rho}(T^+).
\end{equation}
Therefore, $v_{I|J}\overline{\rho}(\tau^+ T^+)$ is an even-primitive vector that is a linear combination of vectors 
$\overline{\pi}_{I|L}=v_{I|J}\overline{\rho}_{I|L}$ with $L$ as above.

In the similar vein, in Subsection \ref{tensorsetup}, using positioning map $P^-$, to a multiindex $(K|L)$ we have assigned a tableaux $T^-$ of shape $\nu/ \omega$ and content $(\la^+/ \mu|0)$ corresponding to $(K|L)$.

Conversely, to a tableau $T^-$ of shape $\nu/ \omega$ and content $(\la^+/ \mu|0)$ we assign the multiindex $(I|J)$, in the following way.
The entries in $J$ are symbols $1\leq j\leq n$, they are weakly increasing, and for each $j$ 
there are exactly $\nu_j-\omega_j$ entries in $J$ that are equal to $j$. 
The entries in $I$ are symbols $1\leq i\leq m$ that are 
entries in $T^-$ listed by rows from top to bottom, in each row ordered from left to right. 
Denote by $Q^-$ the map $T^- \mapsto (I|J)$ defined this way.
In the particular case when the entries in rows of $T^-$ are strictly increasing from left to right and $\la_{I|J}$ is dominant, we obtain that $(I|J)$ is right admissible.

If $Q^-(T^-)=(I|J)$, then the entries in $J$ are weakly increasing, hence we will not obtain all possible multiindexes $(K|L)$ as images under $Q^-$. However, since we are working inside
the exterior algebra $\wedge(Y)$ instead of the tensor algebra over $Y$, after reordering of the terms in $L$ we get $\overline{\rho}_{K|L}=\epsilon\overline{\rho}_{I|J}$, where 
$\epsilon = \pm 1$ and $(I|J)=Q^-(T^-)$ for some $T^-$.
Additionally, define the vector $\overline{\rho}(T^-)=\overline{\rho}_{I|J}$.

If $T'^-$ is obtained from $T^-$ by row permutations, then $\overline{\rho}(T'^-)=\pm \overline{\rho}(T^-)=\pm \overline{\rho}_{I|J}$. Using the
correspondence between $T^-$ and $(I|J)$, we can replace the operator $\sigma^-$ acting on multiindices $(I|J)$ by an operator $\tau^-$ acting on tableaux $T^-$.

For $\sigma\in X^-$ write $\sigma\overline{\rho}_{I|J}=\overline{\sigma\rho_{I|J}}$ and denote by 
$\sigma T^-$ the tableau obtained by applying permutation $\sigma$ to the entries of $T^-$. 
The tableau $\sigma T^-$ corresponds to a multiindex $(K|J)$, where $K$ has the same content as $I$. 
It is clear that 
$\overline{\rho}(\sigma T^-)=\sigma\overline{\rho}(T^-)$.

The operator $\tau^-$ acting on $T^-$ is defined as \[\tau^- T^-=\sum_{\sigma\in X^-} \sigma T^-\]
and the operator $\overline{\rho}\tau^-$ applied to $T^-$ is given as 
\[\overline{\rho}(\tau^- T^-)=\sum_{\sigma\in X^-} \overline{\rho}(\sigma T^-).\]

The expressions $\overline{\rho}\tau^+ T^-$ can be considered as row bipermanents corresponding to the pair of tableaux $T^-_{can}$ and $T^-$ based on formal symbols
 $\rho_{ij}$, see \cite{grs}.

The operators $\tau^-$ and $\sigma^-$ are compatible in the sense that 
\begin{equation}\label{eq2}
\overline{\rho}(\tau^- T^-)=\sigma^-\overline{\rho}(T^-).
\end{equation}
Therefore, $v_{I|J}\overline{\rho}(\tau^- T^-)$ is an even-primitive vector that is a linear combination of vectors 
$\overline{\pi}_{K|J}=v_{I|J}\overline{\rho}_{K|J}$ with $K$ as above.

We illustrate the above definitions on the following example.

\begin{ex}\label{ex1}
Let $m=2$, $n=6$, $\la=(3,3)$ and $\mu=\emptyset$. Then the diagram $[\la'/\mu']$ is of type $(2,2,2)$. Consider the tableau 
$T^+=\begin{matrix}3&4\\5&6\\7&8\end{matrix}$. Then 
\[\tau^+ T^+=\begin{matrix}3&4\\5&6\\7&8\end{matrix}+\begin{matrix}4&3\\5&6\\7&8\end{matrix}+\begin{matrix}3&4\\6&5\\7&8\end{matrix}+
\begin{matrix}3&4\\5&6\\8&7\end{matrix}+\begin{matrix}4&3\\6&5\\7&8\end{matrix}+\begin{matrix}4&3\\5&6\\8&7\end{matrix}+\begin{matrix}3&4\\6&5\\8&7\end{matrix}
+\begin{matrix}4&3\\6&5\\8&7\end{matrix}\]
and 
\[\begin{aligned}\overline{\rho}(\tau^+ T^+)=&\quad\,\rho_{13}\wedge\rho_{15}\wedge\rho_{17}\wedge\rho_{24}\wedge\rho_{26}\wedge\rho_{28}\\
&+\rho_{14}\wedge\rho_{15}\wedge\rho_{17}\wedge\rho_{23}\wedge\rho_{26}\wedge\rho_{28}\\
&+\rho_{13}\wedge\rho_{16}\wedge\rho_{17}\wedge\rho_{24}\wedge\rho_{25}\wedge\rho_{28}\\
&+\rho_{13}\wedge\rho_{15}\wedge\rho_{18}\wedge\rho_{24}\wedge\rho_{26}\wedge\rho_{27}\\
&+\rho_{14}\wedge\rho_{16}\wedge\rho_{17}\wedge\rho_{23}\wedge\rho_{25}\wedge\rho_{28}\\
&+\rho_{14}\wedge\rho_{15}\wedge\rho_{18}\wedge\rho_{23}\wedge\rho_{26}\wedge\rho_{27}\\
&+\rho_{13}\wedge\rho_{16}\wedge\rho_{18}\wedge\rho_{24}\wedge\rho_{25}\wedge\rho_{27}\\
&+\rho_{14}\wedge\rho_{16}\wedge\rho_{18}\wedge\rho_{23}\wedge\rho_{25}\wedge\rho_{27}.
\end{aligned}\]
\end{ex}

\subsection{The operator $\tau$}

Let $(I|J)$, $(K|L)$ and $(M|N)$ be multiindices such that $(K|L)$ is left admissible, $(M|N)$ is right admissible and 
$\overline{\rho}_{I|J}=\epsilon_1 \overline{\rho}_{K|L}=\epsilon_2 \overline{\rho}_{M|N}$, where $\epsilon_1, \epsilon_2 \in \{\pm 1\}$.

There is a unique tableau $R^+$ such that $Q^+(R^+)=(K|L)$ and a unique positioning map $P^+$, such that $P^+(K|L)=R^+$. 
Other tableaux $T^+$ corresponding to $(K|L)$, with respect to different positioning maps $P^+$, differ from $R^+$ only by permutations of entries in its rows.
Then $\overline{\rho}(T^+)=\pm \overline{\rho}(R^+)=\pm\overline{\rho}_{K|L}$.

There is a unique tableau $R^-$ such that $Q^-(R^-)=(M|N)$ and a unique positioning map $P^-$, such that $P^-(M|N)=R^-$. 
Other tableaux $T^-$ corresponding to $(M|N)$, with respect to different positioning maps $P^-$, differ from $R^-$ only by permutations of entries in its columns.
Then $\overline{\rho}(T^-)=\pm \overline{\rho}(R^-)=\pm\overline{\rho}_{M|N}$.

Therefore, any tableaux $T^+$ and $T^-$ constructed as above are closely related to $\overline{\rho}_{I|J}$. 
We will define a repositioning map $Rpos$ that, together with its inverse $Rpos^{-1}$, will fix a correspondence between $T^+$ and $T^-$ and allow us to move from one representation to the other one and back while preserving the correspondence to the expression $\overline{\rho}_{I|J}$.

Fix a tableau $T^+:\mathcal{D}^+ \to \{m+1, \ldots, m+n\}$. We want to assign to $T^+$ a tableau $T^-:\mathcal{D}^-\to \{1, \ldots, m\}$ 
that has the property $\overline{\rho}(T^-) =\pm \overline{\rho}(T^+)$.
This property is satisfied if and only if for each $r$ the $r$-th row of $T^-$ contains exactly those indexes $i$ such that the $i$-th column of $T^+$ contains an entry equal to $m+r$. 

Let a repositioning map $Rpos: \mathcal{D}^+ \to \mathcal{D}^-$ be a bijection that maps each entry $[k,i]\in \mathcal{D}^+$ such that $t^+_{k,i}=m+j$ 
to some $[j,l]\in\mathcal{D}^-$. Then the tableau $T^-=Rp(T^+)$, corresponding to the map $Rpos$, is given by $t^-_{j,l}=i$, where $t^+_{k,i}=m+j$ as above. 
Tableaux $T^-$ like these are in one-to-one correspondence with the maps $Rpos$ satisfying the above property.
Then the basic compatibility requirement 
\[\overline{\rho}(T^+)=\epsilon \overline{\rho}(T^-)=\epsilon \overline{\rho}(Rp(T^+)),\]
where $\epsilon=\pm 1$, is satisfied.

Our next goal is to define an analogue of the operator $\sigma$, defined earlier, acting on $(I|J)$ and $\rho_{I|J}$. We will define an operator $\tau$ acting on a tableau
$T^+$ that will combine the actions of operators $\tau^-$ on $T^+$ and $\tau^-$ on $T^-$. 

Fix a tableau $T^+$, a repositioning map $Rpos$ and $T^-=Rp(T^+)$. If $\sigma\in X^-$, then the tableau $\sigma T^-$ is given 
as $T^-\circ\sigma$, a composition of $\sigma$ and $T^-: \mathcal{D}^- \to \{1, \ldots, m\}$. 
The map $Rp(\sigma):\mathcal{D}^+\to \mathcal{D}^+$ is defined to make the following diagram commutative.
\[\begin{CD} \mathcal{D}^+ @>Rp(\sigma)>> \mathcal{D}^+\\
@VVRposV @VVRposV\\
\mathcal{D}^- @>\sigma>> \mathcal{D}^-
\end{CD}\]
We define $\sigma T^+=T^+\circ Rp(\sigma)$. Then there is the following compatibility condition  
\[\overline{\rho}(\sigma T^+)=\epsilon \overline{\rho}(\sigma T^-).\]
Extending this naturally to linear combinations, we define $\tau^- T^+$ and obtain the following compatibility condition
\begin{equation}\label{eq3}
\overline{\rho}(\tau^- T^+)=\epsilon \overline{\rho}(\tau^- T^-).
\end{equation}

Finally, we define $\tau.T^+= \tau^+\tau^- T^+$.

The following theorem is related to Theorem \ref{p1} and Corollary \ref{korak}.

\begin{teo}\label{p2}
For any tableau $T^+$ of shape $(\la^+/ \mu)'$ and content $(0|\nu /\omega)$, and a repositioning map $Rpos:\mathcal{D}^+\to \mathcal{D}^-$ as above,  
the expression 
$v_{I|J}\overline{\rho}(\tau T^+)$ is a $G_{ev}$-primitive vector of $H^0_G(\la)$.
\end{teo}
\begin{proof}

Denote $T^-= Rp(T^+)$ given by the map $Rpos$. It follows from the proof of Theorem \ref{p1} that 
\[\sigma^-\overline{\rho}(T^-) \equiv 0 \pmod{\prod_{j=1}^{n-1} D^-(m+1, \ldots, m+j)^{\nu_{j+1}-\la^-_j}}.\]
Equations (\ref{eq3}) and (\ref{eq2}) imply that
\[\epsilon \overline{\rho}(\tau^- T^+)= \overline{\rho}(\tau^- T^-) = \sigma^-\overline{\rho}(T^-) \equiv 0 \pmod{\prod_{j=1}^{n-1} D^-(m+1, \ldots, m+j)^{\nu_{j+1}-\la^-_j}}.\]

The expression $\tau^- T^+$ is a linear combination of tableaux $U^+$, and for each such $U^+$ the proof of Theorem \ref{p1} implies
\[\sigma^+ \overline{\rho}(U^+) \equiv 0 \pmod{\prod_{i=1}^{m-1} D^+(1, \ldots, i)^{\la_{i+1}-\mu_i}}.\]

The equation (\ref{eq1}) implies
\[\overline{\rho}(\tau^+ U^+)=\sigma^+ \overline{\rho}(U^+) \equiv 0 \pmod{\prod_{i=1}^{m-1} D^+(1, \ldots, i)^{\la_{i+1}-\mu_i}}.\]
Therefore
\[\overline{\rho}(\tau^+ \tau^- T^+) \equiv 0 \pmod{\prod_{i=1}^{m-1} D^+(1, \ldots, i)^{\la_{i+1}-\mu_i}\prod_{j=1}^{n-1} D^-(m+1, \ldots, m+j)^{\nu_{j+1}-\la^-_j}}\]
showing that $v_{I|J}\overline{\rho}(\tau T^+)$ belongs to $H^0_G(\la)$. It is obvious that 
$v_{I|J}\overline{\rho}(\tau T^+)$ is a $G_{ev}$-primitive vector.
\end{proof}

\subsection{Repositioning map}\label{synch}
For a given tableau $T^+$ there is a number of choices for the map $Rpos:\mathcal{D}^+ \to \mathcal{D}^-$ and the related tableau $T^-=Rp(T^+)$ for which the above theorem 
gives a $G_{ev}$-primitive vector $v_{I|J}\overline{\rho}(\tau T^+)$ of $H^0_G(\la)$.
We would like to fix for every $T^+$ a specific map $Rpos$ and tableau $T^-$ in a way that relates to Yamanouchi words and Littlewood-Richardson tableaux - see section 5.2 of \cite{fulton}.

Let $Q$ be a skew tableau of the shape $\alpha/ \beta$ and $Q^+_{can}$ be the canonical column skew tableau corresponding to the diagram $[\alpha/ \beta]$.
To  $Q$ we assign a word $w=w(Q)$, obtained by reading and concatenating entries in its rows from right to left starting in the top row and proceeding to the bottom row. 
The word $w$ is a {\it lattice word} if in every initial part of the word $w$, the symbol $i$ appears at least as many times as the symbol $i+1$.
The tableau $Q$ is called {\it Yamanouchi} if $w(Q)$ is a lattice word. 
Recall that $Q$ is called {\it semistandard} if all entries in each row are weakly increasing from left to right and all entries in each collumn are strictly increasing from top to bottom.
A {\it Littlewood-Richardson tableau} $Q$ is a tableau that is semistandard and Yamanouchi.

Also, define a word $z=z(Q)=w(Q^+_{can})$, which records instead of the entries in $Q$ their corresponding columns. 
We can think of $z(Q)$ as recording the places of the corresponding letters in $Q$. 
There is a connection of our setup to letter-place algebras defined in \cite{grs} but we will neither need it nor pursue it here.

Now we define the map $Rp$ that sends each tableau $T^+$ of the shape $(\la^+/ \mu)'$ and content $(0|\nu / \omega)$ to 
a tableau $Rp(T^+)=T^-$ of shape $\nu/ \omega$ and content $(\la^+ / \mu|0)$.
For a tableau $T^+$ as above, the word $w(T^+)$ codes letters appearing in $T^+$ which corresponds to the multiindex $J$.
The word $z(T^+)$ codes the corresponding places in $T^+$ and is related to the multiindex $I$.

\begin{df}
The tableau $T^-=Rp(T^+)$ is obtained in the following way. 
When reading the word $w(T^+)$, if the symbol $w_s=m+i$ appears for the $j$-th time, then $t^-_{i,\la^-_i+j}=z_s$, where $z=z(T^+)$.
The map $Rpos:\mathcal{D}^+\to \mathcal{D}^-$ is defined to correspond to this setup.
\end{df}

It is interesting that we can define the counting tableau $C^-$ corresponding to $T^+$ by putting $c^-_{i,\la^-_i+j}=s$ if the symbol $w_s=m+i$ appears for the $j$-th time in $w(T^+)$. While it might be useful for other purposes, we will not need this in what follows.

From this definition it is immediate that $Rpos$ is a bijection.

\begin{df}
Let $T^+$ be a tableau as above. Denote by $w'$ an initial part of the word $w(T^+)$ and for each $i$ denote by $a_{w'}(i)$ the number of appearances of the symbol 
$m+i$ in $w'$. The tableau $T^+$ is called shifted Yamanouchi if for every $w'$ and $i$ we have $\la^-_i+a_{w'}(i)\geq \la^-_{i+1}+a_{w'}(i+1)$.
\end{df}

The meaning of $T^+$ shifted Yamanouchi can be explained equivalently as follows. 
Let $T$ be the tableau corresponding to $T^+$. Define the new word $w_{sh}(T)$ obtained by concatenating of the word $w(L^-_{can})$ first and $w(T^+)$ second.
Then $T^+$ is shifted Yamanouchi if and only if the word $w_{sh}(T)$ is a lattice word. 

Since $w(L^-_{can})$ is a lattice word, the condition that $w_{sh}(T)$ is a lattice word can be expressed as follows.
If $[j,l],[j+1,l]\in \mathcal{D}^-$ correspond to the $l$-th appearance of the symbol $m+j$ and $m+j+1$ in $w_{sh}(T)$, respectively, then 
the $l$-th appearance of $m+j$ in $w_{sh}(T)$ is before the $l$-th appearance of $m+j+1$ in $w_{sh}(T)$.

\begin{ex}\label{ex2}
Let $m=2$, $n=3$, $\la=(\la^+|\la^-)=(2,2|1,1)$, $\mu=(1,0)$, $\nu=(2,2,1)$
and 
$T^+=\begin{matrix}
&4\\
3&5
\end{matrix}$. Then 
$w(T^+)=453$, $z(T^+)=221$ and 
$T^-=
\begin{matrix}
&1\\
&2\\
2
\end{matrix}$.
Also, 
$w_{sh}(T)=45453$ and $T^+$ is shifted Yamanouchi.

\end{ex}

\begin{ex}\label{ex3}
Let $m=n=3$, $\la=(5,4,4|3,2,1)$ and $\mu=(2,2,1)$ which implies $\la'=(6,5,4,3,1)$, $\mu'=(3,2)$. Let $\nu=(5,5,4)$ and $T^+$ be given as 
$\begin{matrix}&&\\&&4\\4&5&5\\5&6&6\\6\end{matrix}$.
Then $w(T^+)=45546656$
and the tableau $T^-$ is 
$\begin{matrix}&&&3&1\\&&3&2&1\\&3&2&1\end{matrix}$. 
Also, $w_{sh}(T)=44455645546656$ and $T^+$ is shifted Yamanouchi.
\end{ex}

\subsection{Further properties of the operator $\tau$}

Let $X^-_j$ be the subgroup of $X^-$ consisting of all elements that permute only the $j$-th column of the diagram $\mathcal{D}^-$, and 
$X^+_i$ be the subgroup of $X^+$ consisting of all elements that permute only the $i$-th row of the diagram $\mathcal{D}^+$.
Then $\tau^-$ decomposes as a product of commuting operators $\tau^-_j$, where $\tau^-_j = \sum_{\sigma\in X^-_j} \sigma$, 
and $\tau^+$ decomposes as a product of commuting operators $\tau^+_i$, where $\tau^+_i = \sum_{\sigma\in X^+_i} \sigma$. 

The following lemma shows a special case when $\overline{\rho}(\tau T^+)$ vanishes.

\begin{lm}\label{zero}
If there are two different entries in the same column of $T^-$ such that the corresponding entries in $T^+$ lie in the same row of $T$, 
then $\overline{\rho}(\tau T^+)=0$. 
\end{lm}
\begin{proof}
Assume that the symbols $i_1$ and $i_2$ appear in positions $(j_1, l)$ and $(j_2, l)$ of $T^-$ and the corresponding entries $m+j_1$ and $m+j_2$ appear in 
positions $(k,i_1)$ and $(k,i_2)$ of $T^+$.

Denote by $\nu^-_l$ the transposition of positions $(j_1,l)$ and $(j_2,l)$ in $\mathcal{D}^-$ and by $\nu^+_k$ the transposition 
of positions $(k,i_1)$ and $(k,i_2)$ of $\mathcal{D}^+$.
Let $X^+_k=\widetilde{X}^+_k \nu^+_k$ be a decomposition of $X^+_k$ as products of $\widetilde{X}^+_k$ and $\nu^+_k$, 
where $\widetilde{X}^+_k$ are representatives of left coset classes of $X^+_k$ by $\nu^+_k$.
%Then every $\sigma^+_k\in X^+_k$ is written as $\sigma^+_k=\widetilde{\sigma}^+_k \nu^+_k$, where $\widetilde{\sigma}^+_k\in \widetilde{X}^+_k$.
Analogously, let $X^-_l=\tau^-_l \widetilde{X}^-_l$ be a decomposition of $X^-_l$ as 
products of $\nu^-_l$ and $\widetilde{X}^-_l$, which consist of representatives of right coset classes of $X^-_l$ by $\nu^-_l$.
%Then every $\sigma^-_l\in X^-_L$ is written as $\sigma^-_l=\nu^-_l\widetilde{\sigma}^-_l$, where $\widetilde{\sigma}^-_l\in \widetilde{X}^-_l$.

Write 
\[\tau=\prod_{i\neq k} \tau^+_k (\sum_{\widetilde{\sigma}^+_k\in \widetilde{X}^+_k} \widetilde{\sigma}^+_k)  \nu^+_k
\nu^-_l (\sum_{\widetilde{\sigma}^-_l\in \widetilde{X}^-_l} \widetilde{\sigma}^-_l) \prod_{j\neq k} \tau^-_j\]
and denote by $Q^-$ a summand in $(\sum_{\widetilde{\sigma}^-_l\in \widetilde{X}^-_l} \widetilde{\sigma}^-_l) \prod_{j\neq k} \tau^-_j T^-$.
We will show that $\overline{\rho}(\nu^+_k \nu^-_l Q^-)=0$.

We have $\nu^-_l Q^- = Q^- + {Q'}^-$, where $(Q')^-$ is obtained from $Q^-$ by switching the entries $i_1$ and $i_2$ at the positions $(j_1,l)$ and $(j_2,l)$.
Let $Q^+$ be such that $Rp(Q^+)=Q^-$. The identity $\rho_{i_1,j_2}\wedge\rho_{i_2,j_1}=-\rho_{i_2,j_1}\wedge\rho_{i_1,j_2}$ shows that 
$\overline{\rho}(\nu^-_l Q^+) = \overline{\rho}(Q^+) - \overline{\rho}((Q')^+)$, where $(Q')^+$ is obtained from $Q^+$ by switching the entries $m+j_1$ and $m+j_2$ at the positions $(k,i_1)$ and $(k,i_2)$. 

Since $\overline{\rho}(\nu^+_k Q^+)=\overline{\rho}(\nu^+_k (Q')^+)=\overline{\rho}(Q^+)+\overline{\rho}((Q')^+)$, we obtain 
$\overline{\rho}(\nu^+_k\nu^-_l Q^-)=0$. Therefore 
$\overline{\rho}(\nu^+_k\nu^-_l(\sum_{\widetilde{\sigma}^-_l\in \widetilde{X}^-_l} \widetilde{\sigma}^-_l) \prod_{j\neq k} \tau^-_j T^+)=0$ and 
we conclude that $\overline{\rho}(\tau T^+)=0$.
\end{proof}

If there are two different entries in the same column of $T^-$ such that the corresponding entries in $T^+$ lie in the same row of $T^+$, then $T^+$ and $T^-$ are called {\it insignificant}. If $T^+$ and $T^-$ are not insignificant, we call them {\it significant}.

\begin{lm}\label{disjoint}
If $S^+$ appears as a term in the expression $\tau T^+$, then $\tau S^+=\tau T^+$.
\end{lm}
\begin{proof}
We will consider the following sequence of tableaux: 
\[\begin{CD}
T^+ @>Rp>> T^- @>\sigma_->> R^- @>(Rp)^{-1}>> R^+@>\sigma_+>> S^+ @>Rp>> S^-,
\end{CD}\]
where $\sigma_-\in X^-$ and $\sigma_+\in X^+$.
We will first assume that $\sigma_-\in X^-_l$ and then combine such permutations to a general $\sigma_-\in X^-$ later.

Keeping in mind the compatibility of maps $Rp$, $Rp^{-1}$ and actions of $\sigma_-$ and $\sigma_+$, 
we describe the entries in the $l$-th column of each tableau $T^-, R^-, S^-$ and the corresponding entries in $T^+, R^+, S^+$ as follows.

 Let 
$[j_1, l], \ldots, [j_u,l] \in \mathcal{D}^-$ are entries in the $l$-th column of $\mathcal{D}^-$, and corresponding to these there are entries
$[k_1,i_1], \ldots, [k_u,i_u] \in \mathcal{D}^+$ such that $t^+_{k_t,i_t}=m+j_t$ and $t^-_{j_t,j}=i_t$ for each $t=1, \ldots, u$.

The action of $\sigma_-$ is given as $\sigma_-[j_t,l]=[\sigma_-(j_t),l]=[j_{\sigma_-(t)},l]$ for each $t=1, \ldots, u$, where the action on indices $j_1, \ldots, j_u$ (denoted also by $\sigma_-$) is induced by $\sigma_-$.
Then $r^-_{\sigma_-(j_t),l}=t^-_{j_t,l}=i_t$ and $r^+_{k_t,i_t}=m+\sigma_-(j_t)$ for each $t=1, \ldots, u$.

The permutation $\sigma_+$ sends $[k_t,i_t]$ to $[k_t, \sigma_+(i_t)]=[k_t,i_{\sigma_+(i_t)}]$ for each $t=1, \ldots, u$, where the action on indices $i_1, \ldots, i_u$ (denoted also by $\sigma_+$) is induced by $\sigma_+$.
Then $s^+_{k_t,\sigma_+(i_t)}=r^+_{k_t,i_t}=m+\sigma_-(j_t)$ and $s^-_{\sigma_-(j_t),l}=\sigma_+(i_t)$ for each $t=1, \ldots, u$.
This compares to $t^+_{k_t,i_t}=m+j_t$ and $t^-_{j_t,j}=i_t$ for each $t=1, \ldots, u$.

Now consider the sequence of tableaux:
\[\begin{CD}
S^+ @>Rp>> S^- @>\sigma^{-1}_->> P^- @>(Rp)^{-1}>> P^+@>\sigma^{-1}_+>> Q^+ @>Rp>> Q^-.
\end{CD}\]
Using the above formulae, we obtain $q^+_{k_t,i_t}=m+j_t$ and $q^-_{j_t,l}=i_t$ for each $t=1, \ldots, u$, showing that $Q^+=T^+$ and $Q^-=T^-$.

This implies that $S^+$ appears as a term in the expression $\tau T^+$ if and only if $T^+$ appears as a term in the expression $\tau S^+$. Therefore, in this case $\tau S^+=\tau T^+$.
\end{proof}

\section{Operator $\tau$ and its action on Littlewood-Richardson tableaux}

Let us recall that the Littlewood-Richardson tableaux $T$ of shape $\la'/ \mu'$ and content $(0|\nu)$ play an important role later because they count a number of even-primitive vectors in the simple supermodule $L_G(\la)$. 

Instead of tableaux $T$, we will work with tableaux $T^+$ of shape $(\la^+/ \mu)'$.

\subsection{Clausen column and row preorders}

For a tableau $T^+$, and for every index $j$ corresponding to a column of $T^+$ and a number $m+1\leq k\leq m+n$, we define
$c_{jk}$ to be the number of occurences of symbols $\{m+1, \ldots, m+k\}$ in the columns of $T^+$ of index $j$ or higher.
We organize these entries in a form of a Clausen column matrix $C(T^+)=\begin{pmatrix}c_{jk}\end{pmatrix}$.

Additionally, for every index $i$ corresponding to a row of $T$ and a number $m+1\leq k\leq m+n$, we define 
$r_{ik}$ to be the number of occurences of symbols $\{m+1, \ldots, m+k\}$ in the rows of $T^+$ of index $i$ or lower.
We organize these entries in a form of a Clausen row matrix $R(T^+)=\begin{pmatrix}r_{ik}\end{pmatrix}$.

\begin{ex}\label{ex4}
Let $m=n=3$, $\la=(5,4,4|3,2,1)$, $\mu=(2,2,1)$ and $T^+$ be a tableau 
$\begin{matrix}&&\\&&4\\4&5&5\\5&6&6\\6\end{matrix}$  as in Example \ref{ex3}. Then the Clausen column matrix
%$C(T)=\begin{pmatrix}5&4&4&3&2&1\\5&4&3&2&1&0\\4&3&2&1&0&0\end{pmatrix}.$
$C(T^+)=\begin{pmatrix}2&1&1\\5&3&2\\8&5&3\end{pmatrix}$
and Clausen row matrix 
$R(T^+)=\begin{pmatrix}1&1&1\\2&4&4\\2&5&7\\2&5&8\end{pmatrix}$.
\end{ex}

It is clear that if $T'^+$ is obtained from $T^+$ by permuting entries in the same column, then $C(T'^+)=C(T^+)$, hence the Clausen column matrix 
can be defined for column tabloids, the equivalence classes of tableaux with respect to permutations of entries within columns.
Analogously, if $T'^+$ is obtained from $T^+$ by permuting entries in the same row, then $R(T'^+)=R(T^+)$, hence the Clausen row matrix 
can be defined for row tabloids, the equivalence classes of tableaux with respect to permutations of entries within rows.

The Clausen column preorder $\prec_c$ on the set of tableaux of the same skew shape is defined as follows.
Let $T^+$ and $T'^+$ be of the shape $(\la^+/ \mu)'$, $C(T^+)=(c_{jk})$ and $C(T'^+)=(c'_{jk})$. Then 
$T^+\prec_c T'^+$ if and only if $C(T^+)=C(T'^+)$ or for some $j$ and $k$ we have $c_{il}=c'_{il}$ for all $i>j$ and all $l=1, \ldots n$;
$c_{jl}=c'_{jl}$ for all $l<k$ and $c_{jk}<c'_{jk}$.
If $T^+\prec_c T'^+$ and $C(T^+)\neq C(T'^+)$, then we will write $T^+<_c T'^+$.

The Clausen row preorder $\prec_r$ on the set of tableaux of the same skew shape is defined as follows.
Let $T^+$ and $T'^+$ be of the shape $(\la^+/ \mu)'$, $R(T^+)=(r_{ik})$ and $R(T'^+)=(r'_{ik})$. Then 
$T^+\prec_r T'^+$ if and only if $R(T^+)=R(T'^+)$ or for some $i$ and $k$ we have $r_{jl}=r'_{jl}$ for all $j<i$ and all $l=1, \ldots n$;
$r_{il}=r'_{il}$ for all $l<k$ and $r_{ik}<r'_{ik}$.
If $T^+\prec_r T'^+$ and $R(T^+)\neq R(T'^+)$, then we will write $T^+<_r T'^+$.

\begin{lm}\label{l1}
The restriction of the Clausen preorder $\prec_c$, and $\prec_r$ respectively, to the set of semistandard tableaux of the skew shape $(\la^+/ \mu)'$ is a linear order. 
Consequently, the restriction of the Clausen preorders $\prec_c$ and $\prec_r$ to the set of Littlewood-Richardson tableaux of the same shape is a linear order.
\end{lm}
\begin{proof}
Let $T^+$ and $T'^+$ be semistandard tableaux of the shape $(\la^+/ \mu)'$ with entries in the set $m+1, \ldots, m+n$; $T^+\prec_c T'^+$ and $T'^+\prec_c T^+$. 
Then $C(T^+)=C(T'^+)$, and since the entries in all columns are strictly increasing, we infer $T^+=T'^+$ and obtain an order on semistandard tableaux.
Since $\prec_c$ is linear order on tableaux with different Clausen matrices, the claim for semistandard tableaux and Littlewood-Richardson tableaux follows.
Analogous arguments work for the Clausen preorder $\prec_r$.
\end{proof}

\subsection{Linear independence of even-primitive vectors}

Assume that $\la$ is a hook partition and $\mu<\la$.
Denote by $LR((\la^+)'/ \mu',\nu)$ the set of all Littlewood-Richardson tableaux $T^+$ of shape $(\la^+)'/ \mu'$ and content $(0|\nu)$, where
$\nu$ is a partition containing $\omega$.

\begin{lm}\label{order}
Assume that $T^+$ is semistandard and shifted Yamanouchi. Let $i_1$ and $i_2$ be symbols appearing in the same column of $T^-$, 
say at its $j_1$-th and $j_2$-th row respectively, where $j_1<j_2$. If the corresponding symbols $m+j_1$ appear at the $k_1$-th row of $T^+$ and the corresponding symbol 
$m+j_2$ appears at the $k_2$-th row of $T^+$, then $k_1<k_2$.
\end{lm}
\begin{proof}
Assume that the symbols $i_1$ and $i_2$ appear in positions $(j_1, l)$ and $(j_2, l)$ of $T^-$, respectively. 
Corresponding to this, symbols $m+j_1$ and $m+j_2$ appear at positions $(k_1,i_1)$ and $(k_2,i_2)$ in $T^+$.
Since $T^+$ is shifted Yamanouchi, we infer that $k_1\leq k_2$, and if $k_1=k_2$, then $i_2<i_1$. 
If $k_1=k_2$, then $T^+$ semistandard implies $m+j_2\leq m+j_1$, which is a contradiction. Therefore $k_1<k_2$.
\end{proof}

In the representation theory of Schur algebra, see for example sections 2.4 and 2.5 of \cite{martin}, an important role is played by bideterminants and their straigtening formula stating that every bideterminant is a linear combination of bideterminants based on a pair of semistandard tableaux of the same shape.

Analogous result is valid in the setting of the four-fold (or letter-place) algebra of \cite{grs}. A particular case of the straigtening formula (Theorem 8 of \cite{grs}) states 
that every bipermanent is a linear combination of bipermanents based on a pair of standard tableaux.

In our case, instead of a pair or (semi)standard tableaux, we will deal with a pair of related tableaux $(T^+, T^-)$ and the basis we will construct corresponds to the case when 
$T^+$ is semistandard and $T^-$ is anti-semistandard - see below. As a particular case of these pairs we will obtain Littlewood-Richardson tableaux. 

\begin{df}
A tableau is called anti-semistandard if the entries in its rows are strictly decreasing from left to right and entries in its columns are weakly decreasing from top to bottom.
\end{df}

\begin{lm}\label{rowsdecrease}
If $T^+$ is semistandard, then that entries in each row of $T^-$ are strictly decreasing from left to right.
\end{lm}
\begin{proof}
Let $[j,l]$ and $[j,l+1]$ be two consecutive entries in the $j$-th row of $\mathcal{D}^-$, and $[k_1,i_1]$ and $[k_2,i_2]$ be the corresponding entries in $\mathcal{D}^+$, such that 
$t^+_{k_1,i_1}=m+j$ is the $l$-appearance and $t^+_{k_2, i_2}=m+j$ is the $(l+1)$-st appearance of $m+j$ in $w_{sh}(T)$. This implies $k_1\leq k_2$, 
and if $k_1=k_2$, then $i_1>i_2$.

Assume now that $k_1<k_2$ and $i_1\leq i_2$. Then the position $[k_2,i_1]\in \mathcal{D}^+$, and $T^+$ semistandard implies
$T^+_{k_1,i_1}=m+j < t^+_{k_2,i_1} \leq t^+_{k_2,i_2} =m+j$, which is a contradiction. Therefore $i_1>i_2$ and this 
shows that entries in the rows of $T^-$ are strictly decreasing from left to right.
\end{proof}

\begin{lm}\label{strict}
If $T^+$ semistandard and shifted Yamanouchi, then $T^-$ is anti-semistandard.
\end{lm}
\begin{proof}
By Lemma \ref{rowsdecrease}, entries in the rows of $T^-$ are strictly decreasing from left to right.

Next, let $[j, l]$ and $[j+1,l]$ be two consecutive entries in the $l$-th column of $\mathcal{D}^-$ and $[k_1,i_1]$ and $[k_2,i_2]$ be the corresponding entries in $\mathcal{D}^+$ such that $t^+_{k_1,i_1}=m+j$ is the $l$-appearance of $m+j$ and $t^+_{k_2, i_2}=m+j+1$ is the $l$-st appearance of $m+j+1$ in $w_{sh}(T)$. Then $k_1< k_2$ by Lemma \ref{order}.

We claim that $i_1\geq i_2$. Assume to the contrary that $i_1<i_2$ and assume that index $i_{2}'$ is maximal such that $[k_2,i_2']\in\mathcal{D}^+$ and 
$t_{k_2, i_2'}=m+j+1$.
Then the position $[k_1, i_2']$ belongs to the diagram $\mathcal{D}^+$. Since $T^+$ is semistandard, we must have $t^+_{k_2,a}=m+j+1$ for 
$a=i_2, \ldots, i_2'$; $k_2=k_1+1$ and $t^+_{k_1, b}=m+j$ for each $b=i_1, \ldots, i_2'$. 
The initial part of $w_{sh}(T)$ that ends at the position $[k_1,i_1+(i_2'-i_2+1)]\in \mathcal{D}^+$ has the last symbol $m+j$ and it contains the same number 
$l-(i_2'-i_2+1)$ of symbols $m+j$ as $m+j+1$. 
This is a contradiction with the assumption that $w_{sh}(T)$ is a lattice word.

Therefore entries in the columns of $T^-$ are weakly decreasing from top to bottom and $T^-$ is anti-semistandard.
\end{proof}

\begin{pr}\label{symmetry}
Assume $T^+$ is semistandard. Then $T^+$ is shifted Yamanouchi if and only if $T^-$ is anti-semistandard.
\end{pr}
\begin{proof}
The necessary condition is established in Lemma \ref{strict}.

For the sufficient condition, assume that $T^+$ is semistandard and $T^-$ is anti-semistandard. 
Let $[j_1,l]$ and $[j_2,l]$ belong to $\mathcal{D}^-$ and $j_1<j_2$. Then $t^-_{j_1,l}=i_1\geq i_2=t^-_{j_2,l}$ since $T^-$ is anti-semistandard.
Denote by $[k_1,i_1]$ and $[k_2,j_2]$ the corresponding elements of $\mathcal{D}^+$ such that $t^+_{k_1,i_1}=m+j_1$ and $t^+_{k_2,i_2}=m+j_2$.

If $k_1=k_2$, then $i_1>i_2$ and $m+j_1\leq m+j_2$ because $T^+$ is semistandard. This means that the smaller entry $m+j_1$ appear to the right of the bigger entry $m+j_1$ in the same row of $T^+$, hence the lattice condition in $w_{sh}(T)$ is satisfied for this pair.

If $k_1>k_2$, then $[k_1,i_2]\in \mathcal{D}^+$ and $T^+$ semistandard implies $m+j_2=t^+_{k_2,i_2}<t^+_{k_1,i_2}\leq t^+_{k_1,i_1}=m+j_1$, which is a contradiction.

If $k_1<k_2$, then the smaller entry $m+j_1$ appears in the higher row of $T^+$ than the bigger entry $m+j_2$, hence the lattice condition in $w_{sh}(T)$ is satisfied for this pair.

Therefore $T^+$ is shifted Yamanouchi.
\end{proof}

\begin{df}
If $T^+$ is semistandard and the corresponding $T^-$ is anti-semistandard, then $T^+$ is called marked. The set of all marked tableaux $T^+$ of the shape $(\la^+/ \mu)'$ and content 
$(0|\nu/\omega)$ is denoted by $M((\la^+)'/ \mu',\nu/\omega)$.
\end{df}

We will analyze tableaux $T'$ appearing in the expression $\tau T^+$ for $T^+$ marked. 

\begin{lm}\label{l2}
Assume $T^+\in M((\la^+)'/ \mu',\nu/\omega)$. If $T'^+$ appears in the expression for $\tau^- T^+$, then $T'^+\prec_c T^+$ and $T'^+\prec_r T^+$.
The tableau $T^+$ appears in $\tau^- T^+$ with the coefficient one.
\end{lm}
\begin{proof}
Denote the entries in the $l$-th column of $\mathcal{D}^-$, listed from top to bottom by $[j_1,l], \ldots, [j_s,l]$ (so that $j_t=j_1+t-1$ for each $1\leq t \leq s$), and by $\sigma\in X^-$ a permutation of these entries. Also, denote by
$[k_1, i_1], \ldots, [k_s,i_s]$ the corresponding entries in $\mathcal{D}^+$. 
Then the entry at the position $[k_t,i_t]$ in the tableau $T^+$ is the $l$-th appearance of $m+j_t$ in $w_{sh}(T)$, for each $t=1, \ldots, s$.
Then $i_1\geq i_2\geq \ldots \geq i_s$ since $T^-$ is anti-semistandard. Also, Lemma \ref{order} shows $k_1<k_2<\ldots <k_s$. 

Let $\sigma\in X^-$ be arbitrary such that $\sigma\neq 1$, and let $k$ be the smallest index such that there is a position $[k,i]\in \mathcal{D}^+$ that is been moved by $\sigma$. Then those entries in the $k$-th row of $T^+$ that are replaced in $T'^+=\sigma T^+$ (and there is at least one) are replaced by entries that are higher than the original entries 
(because $k_t<k_u$ implies $t^+_{k_t,i_t}=m+j_t<m+j_u=t^+_{k_u,i_u}$). Therefore, $T'^+<_r T^+$ and the tableau $T^+$ appears in $\tau^- T^+$ with the coefficient one.

If $\sigma\in X^-$ only permutes entries in the same column of $T^+$, then $C(T^+)=C(T'^+)$, where $T'^+=\sigma T^+$. In this case $T'^+\prec_c T^+$.
Otherwise, let $i$ be the highest index such that there is an entry in the $i$-th column of $T^+$ that is been moved by $\sigma$ to a different column. 
Then all entries in the $i$-th column of $T^+$ either remain the same or (on at least one occasion) are replaced in $T'^+$ by entries that are higher than the original entries 
(because $i_t > i_u$ implies $t^+_{k_t,i_t}=m+j_t<m+j_u=t^+_{k_u,i_u}$). Therefore $T'^+<_c T^+$.
\end{proof}

Because of the requirements we have imposed on tableaux $T$ and $T^{opp}$, it is clear that there is a one-to-one correspondence between $T$ and $T^+$ and between 
$T^{opp}$ and $T^-$. We will denote $\overline{\rho}(T)=\overline{\rho}(T^+)$, $\overline{\rho}(T^{opp})=\overline{\rho}(T^-)$,
$\overline{\rho}(\tau T)=\overline{\rho}(\tau T^+)$ and so on.
To a tableau $T^+$ we have assigned a pair of multiindices $I|J$ and vectors $v_{I|J}$, $\rho_{I|J}$ and $\overline{\rho}_{I|J}$. Let us denote $cont(T^+)=cont(I|J)$ and $v^+_T=v_{I|J}$.
It is clear that $cont(I|J)=cont(K|L)$ implies $v_{I|J}=v_{K|L}$. Since all tableaux $T'^+$ appearing in the expression for $\tau T^+$ have the same content as $T^+$, 
we have $v^+_{T'}=v^+_T$ and the vector $v^+_T\rho(\tau T^+)$, which is a linear combination of primitive vectors $\pi_{I|L}$, is an even-primitive vector.
Analogously, $v^+_T\overline{\rho}(\tau T^+)$, which is a linear combination of primitive vectors $\overline{\pi}_{I|L}$, is an even-primitive vector.

\begin{teo}\label{jeden} 
The vectors $\overline{v}(T^+)=v^+_T\overline{\rho}(\tau T^+)$ for $T^+\in M((\la^+)'/ \mu',\nu/\omega)$ are linearly independent over a ground field $K$.
\end{teo}
\begin{proof}
It is enough to show that vectors $\overline{\rho}(\tau T^+)$ for $T^+\in M((\la^+)'/ \mu',\nu/\omega)$ are linearly independent.
We will show that $\tau T^+$ has $T^+$ as its leading element with respect to the preorder $\prec_r$ and all other semistandard tableaux $Q^+$ appearing in the expression for $\tau T^+$satisfy $Q^+<_r T^+$.

Consider $\sigma\in X^+$. If $\sigma$ interchanges only identical entries in $T^+$, then $\sigma T^+ = T^+$. If $\sigma T^+ \neq T^+$, then $\sigma T^+$ is not semistandard and 
$\sigma T^+<_r T^+$. 
It follows from Lemma \ref{l2} that all nontrivial summands $T'^+$ in $\tau^- T^+$ (those $T'^+=\sigma T^+$ for $\sigma\in X^-$ such that $\sigma\neq 1$) satisfy $T'^+<_r T^+$.
Since $\sigma \in X^+$ permutes the rows of $T^+$ and $T'^+$, all tableaux $T''^+$ appearing as summands in $\tau^+ T'^+$ for $T'^+<_r T^+$ also satisfy $T''^+ <_r T^+$.
Therefore all summands $T''^+$ of $\tau T^+$ either equal to $T^+$, or otherwise $T''^+<_r T^+$ (and $T''^+$ is not semistandard). In any case $T''^+\prec_r T^+$.

According to Theorem 4.4 of \cite{t2} (see also Section 5 of \cite{t1} or Section 5.7 of \cite{bu}), there is a basis for bipermanents of fixed skew shape and content given by bipermanents corresponding to semistandard tableaux. This statement is a generalization of the classical result of \cite{grs}.

We can express $\overline{\rho}(\tau T^+)$ as 
\[\overline{\rho}(\tau T^+)=\sum_{Q^+ semistandard} \alpha_{Q^+} \overline{\rho}(\tau^+ Q^+),\]
a linear combination of the basis elements $\overline{\rho}(\tau^+ Q^+)$, where $Q^+$ are semistandard tableaux of 
the shape $(\la^+/ \mu)'$. 
Due to the above observations, $\alpha_{T^+}=1$ and $\alpha_{Q^+}\neq 0$ implies $Q^+<_r T^+$.
Therefore the terms in this expression for $\overline{\rho}(\tau T^+)$ are ordered with respect to $\prec_r$ and the highest term is 
$\overline{\rho}(\tau^+ T^+)$. 

Since $\prec_r$ restricted on semistandard tableaux is a linear order by Lemma \ref{l1}, we conclude that the expressions $\overline{\rho}(\tau T^+)$ for 
$T^+\in M({\la'}^+/ \mu',\nu/\omega)$ are linearly independent over a ground field $K$ of characteristic zero.
\end{proof}
\begin{rem}
The above theorem could be proved analogously using the ordering $\prec_c$ instead of $\prec_r$. 
\end{rem}

\subsection{$H^0_G(\la)$ irreducible} 
The indecomposable universal highest weight modules $V(\la)$ over classical Lie superalgebras were investigated by Kac in \cite{kac2, kac3}. He proved that 
$V(\la)$ is irreducible if and only if the weight $\la$ is {\it typical}. Using the contravariant duality given in Proposition 5.8 of \cite{z} that is induced by the supertransposition $\tau: c_{ij}\mapsto (-1)^{|i|(|j|+1)} c_{ji}$, we infer that the induced module $H^0_G(\la)$ is irreducible if and only if $\la$ is typical.

By Theorem 3.20 of \cite{br}, a weight $(\mu|\nu)$ is polynomial if and only if it is given by a hook partition $\la$.
If $\la$ is a hook partition, then by Theorem 4.1 of \cite{fm2}, $H^0_G(\la)$ is irreducible if and only if $\la^+_m\geq n$. 
In this case the even-primitive vectors in $H^0_G(\la)=L_G(\la)$ are in bijective correspodence with Littlewood-Richardson tableaux $LR(\la'/\mu',\nu)$ as was mentioned earlier.

We will show that if $\la^+_m\geq n$, then there is a bijective correspondence between $M((\la^+)'/ \mu',\nu/\omega)$ and $LR(\la'/ \mu',\nu)$. 

\begin{lm}\label{ss}
If $\la^+_m \geq n$, then $T^+$ semistandard implies $T$ semistandard. 
\end{lm}
\begin{proof}
If $T^+$ is semistandard, then entries in $T$ are increasing down each column. Assume that the last column of $T^+$ has an index $i\leq m$.
If $i<m$, then the second part of the tableau $T$ ($L^-_{can}$ of shape $\omega$) splits off from the first part $T^+$ of the tableau $T$, and then $T$ is automatically semistandard. Hence, assume $i=m$.
 Then its lowest entry appears at the 
$\la^+_i$-th row and its value is at most $m+n$. Since $\la^+_i\geq \la^+_m\geq n$, and entries in the $i$-th column of $T^+$ are increasing, then any entry in $T^+$ in the 
$k$-th row and $i$-th column has to be smaller or equal to $m+k$. Since the entry in the $k$-th row and the $(m+1)$-st column of $T$ (if any) equals to $m+k$, we conclude that $T$ is semistandard.
\end{proof}

The following is an example when $T^+$ is semistandard but $T$ is not.

\begin{ex}
Let $\la=(2,1|1,0)$, $\mu=(1,0)$ and $\nu=(2,1)$ in $GL(2|2)$. Then
$T=\begin{matrix}
&4&3\\
3
\end{matrix}$
is not semistandard while
$T^+=\begin{matrix}&4\\3 \end{matrix}$ is semistandard.
\end{ex}

Analogously to the map $Rp$ defined in \ref{synch}, we define a map $Opp$ which sends each tableau $T$ of shape $(\la/ \mu)'$ to a tableau $T^{opp}$ of the shape $\nu$
as follows.

\begin{df}
The tableau $T^{opp}=Opp(T)$ is obtained in the following way. 
When reading the word $w(T)$, if the symbol $w_s=m+i$ appears for the $j$-th time, then $t^{opp}_{i,\la^-_i+j}=z_s$, where $z=z(T)$.
The map $Oppos:\mathcal{D}\to \mathcal{D}^{opp}$ is defined to correspond to this setup.
\end{df}

\begin{ex}\label{ex5}
In the setup of Example \ref{ex2} we have 
$T=\begin{matrix}
&4&3\\
3&5&4
\end{matrix}$,
$w(T)=34453$, $z(T)=32321$,
and 
$T^{opp}=
\begin{matrix}
3&1\\
2&3\\
2
\end{matrix}$.
\end{ex}

The previous example shows that we cannot expect in general that the tableau $T^-$ is a subtableau of $T^{opp}$. 
However, such a property is important in relation to Littlewood-Richardson tableaux of shape $(\la/ \mu)'$.

\begin{df}
Tableaux $T$ and $T^{opp}$ are called behaved if the following equivalent diagrams are commutative, where the vertical arrows are natural inclusions.
\[\begin{CD} T @>Opp>> T^{opp}\\
@AAA @AAA\\
T^+ @>Rp>> T^-
\end{CD}\qquad \qquad 
\begin{CD} \mathcal{D}@>Oppos>> \mathcal{D}^{opp}\\
@AAA @AAA\\
\mathcal{D}^+ @>Rpos>> \mathcal{D}^-
\end{CD}\]
\end{df}

Next Lemma shows that there is a large class of tableaux that are behaved.

\begin{lm}\label{beh}
If a tableau $T$ as above is semistandard, then it is behaved.
\end{lm}
\begin{proof}
Assume $T$ is semistandard and $k$ is such that $\la^-_k>0$. Since the first $k$ entries in the $m+1$-st column of $T$ are $m+1, \ldots, m+k$, $T$ semistandard implies that 
the first $\la^-_k$ appearances of symbol $m+k$ are in the $k$-th row of the second part of $T$, corresponding to $\la^-$. These appearances then transfer to the $k$-th row of 
the second part of $T^{opp}$ corresponding to $\omega$. Since this is true for all rows of $\la^-$, the tableau $T$ is behaved.
\end{proof}

\begin{ex}\label{ex6} In the setup of Example \ref{ex3}, 
$T= 
\begin{matrix}&&&4&4&4\\&&4&5&5\\4&5&5&6\\5&6&6\\6\end{matrix}$, \linebreak $w(T)=44455465546656$ and  
$T^{opp}=
\begin{matrix}6&5&4&3&1\\5&4&3&2&1\\4&3&2&1\end{matrix}$.
Hence $T$ is behaved.
\end{ex}

The following examples show that there is no particular correlation between properties $T^+$ semistandard and $T^+$ shifted Yamanouchi (even for behaved tableaux).

\begin{ex}
Let $\la=(3,2,2|1,1,0)$, $\mu=(2,1)$ and $\nu=(3,2,1)$ in $GL(3|3)$. 
Let $T=\begin{matrix}
&&6&4\\
&4&4&5\\
5&
\end{matrix}$.
Then 
$T^+=\begin{matrix}
&&6\\
&4&4\\
5&
\end{matrix}$,
$T^{opp}=
\begin{matrix}
4&3&2\\
4&1\\
3
\end{matrix}$
and 
$T^-=
\begin{matrix}
&3&2\\
&1\\
3
\end{matrix}$, 
showing that $T$ is behaved.
Then $T^+$ is not semistandard but $T^+$ is shifted Yamanouchi since the second appearance of symbol $4$ is before the second appearance of $5$. 
$T$, however, is not Yamanouchi 
because symbol $6$ appears before symbol $5$.
\end{ex}

\begin{ex}
Let $\la=(2,2|1,0,0)$, $\mu=(0,0)$ and $\nu=(2,1,1,1)$ in $GL(2|4)$. 
Let $T=\begin{matrix}
4&6&4\\
5&7&
\end{matrix}$.
Then
$T^+=\begin{matrix}
4&6\\
5&7
\end{matrix}$,
$T^{opp}=
\begin{matrix}
3&1\\
1&\\
2&\\
2&
\end{matrix}$
and 
$T^-=
\begin{matrix}
&1\\
1&\\
2&\\
2&
\end{matrix}$ showing that $T$ is behaved.
Then $T^+$ is semistandard and $T^+$ is not shifted Yamanouchi because $6$ appears before $5$.
\end{ex}

Nevertheless, there is the following relationship.

\begin{lm}\label{ssY}
If $T$ is semistandard and $T^+$ is shifted Yamanouchi, then $T$ is Yamanouchi.
\end{lm}
\begin{proof}
In order to show that $T$ is Yamanouchi we only need to verify the following. If the entries $i_1$ and $i_2$ appear at the position $(j_1, l)$ and $(j_2, l)$ of $T^{opp}$
such that $j_1<j_2$ and $l\leq \la^-_{j_1}$ (meaning that the position $(j_1,l)$ belongs to $\la^-$-part of $T^{opp}$), then the $l$-th appearance of $j_1$ in $w(T)$ comes before the
$l$-th appearance of $j_2$ in $w(T)$.

Since $T$ is semistandard, its first $j-1$ rows can only contain entries that do not exceed $m+j-1$. Since $l\leq l^-_j$, the first $l$ entries in the $j$-th row of $T$, read from right to left, are equal to $m+j$. Therefore, all the entries that are larger than $m+j$ appear only after the $l$-th appearance of the symbol $m+j$ in $w(T)$.
\end{proof}

\begin{pr}\label{compare}
If $\la^+_m\geq n$, then $T\in LR(\la'/ \mu',\nu)$ if and only if $T^+\in M((\la^+)'/ \mu',\nu/\omega)$.
\end{pr}
\begin{proof}
It is clear that $T$ semistandard implies $T^+$ semistandard.
Since $T$ is behaved by Lemma \ref{beh}, the tableau $T^-$ is a subtableau of $T^{opp}$. Therefore, $T$ Yamanouchi implies $T^+$ shifted Yamanouchi. 
This shows that $T\in LR(\la'/ \mu',\nu)$ implies $T^+\in M((\la^+)'/ \mu',\nu/\omega)$.

Conversely, if $T^+$ is semistandard, then $T$ is semistandard by Lemma \ref{ss}. Then Lemma \ref{ssY} shows that $T^+$ shifted Yamanouchi implies $T$ Yamanouchi.
Therefore, 
$T^+\in M((\la^+)'/ \mu',\nu/\omega)$ implies $T\in LR(\la'/ \mu',\nu)$.
\end{proof}

The following example shows that the assumption $\la^+_m\geq n$ cannot be removed from Proposition \ref{compare}.

\begin{ex}
Let $\la=(2,2|1,1,0)$, $\mu=(0,0)$ and $\nu=(2,2,1,1)$ in $GL(2|4)$. 
Let $T=\begin{matrix}
4&6&4\\
5&7&5
\end{matrix}$.
Then $T^+=\begin{matrix}
4&6\\
5&7
\end{matrix}$,
$T^{opp}=
\begin{matrix}
3&1\\
3&1\\
2&\\
2&
\end{matrix}$
and 
$T^-=
\begin{matrix}
&1\\
&1\\
2&\\
2&
\end{matrix}$ showing that $T$ is behaved.
Then $T^+$ is semistandard and shifted Yamanouchi but $T$ is neither semistandard nor Yamanouchi (because $6$ appears before $5$).
\end{ex}

If $\la$ is a hook partitition such that $\la^+_m\geq n$, then a complete desription of $G_{ev}$-primitive vectors in $H^0_G(\la)$ is given in the following theorem.

\begin{teo}\label{LR}
Assume $\la$ is a hook partition and $H^0_G(\la)$ is irreducible.
The set of all elements $\overline{v}(T^+)=v^+_T\overline{\rho}(\tau T^+)$ for $T^+\in M((\la^+)'/ \mu',\nu/\omega)$, form a basis of all even-primitive vectors of weight $(\mu|\nu)$ in $H^0_G(\la)$.
\end{teo} 
\begin{proof}
If $H^0_G(\la)$ is irreducible, then the multiplicity of even-primitive vectors on $H^0_G(\la)=L_G(\la)$ of weight $(\mu|\nu)$ equals the Littlewood-Richardson coefficient $C^{\la'}_{\mu',\nu}$ by Theorem 6.11 of \cite{br}. 
This coefficient equals the cardinality of elements of the set $LR(\la'/ \mu',\nu)$ 
by Proposition 5.3 of \cite{fulton}.
By Proposition \ref{compare}, this is the same as the cardinality of elements of the set $M((\la^+)'/ \mu',\nu/\omega)$.

Therefore, using Theorem \ref{jeden}, we conclude that vectors $\overline{v}(T^+)$ for $T^+$ from $ M((\la^+)'/ \mu',\nu/\omega)$ form a basis of all even-primitive vectors of $H^0_G(\la)$ of the weight $(\mu|\nu)$.
\end{proof}

Theorem \ref{LR} gives a satisfactory description of even-primitive vectors in $H^0_G(\la)$ in the case when $\la$ is hook partition and $\la_m\geq n$. In the next section we will remove the condition $\la_m \geq n$ and together with the induced supermodule $H^0_G(\la)$ we will also consider its subsupermodule $\nabla_G(\la)$.

According to \cite{zubcom}, over a ground field of characteristic different from $2$, we have $\nabla_G(\la)=H^0_G(\la)$ if and only if $\lambda_m\geq n$ and $\lambda_{m+n}\geq 0$. If $char(K)=0$, then this implies $H^0_G(\la)=L(\la)$ is irreducible, the case investigated above.

Finally, we can describe even-primitive vectors in an arbitrary induced supermodule $H^0_G(\la)$ using the following observation.

\begin{pr}
Let $\la=(\la_1, \ldots, \la_m; \la_{m+1}, \ldots, \la_{m+n})$ be a dominant weight. Denote $\alpha^+=(1, \ldots, 1; 0, \ldots, 0)$ and $\alpha^-=(0, \ldots, 1;1, \ldots, 1)$.
Then there are nonnegative integers $a^+$, $a^-$ such that for the weight $\mu=\la+a^+ \alpha^+ + a^- \alpha^-$ there is $\nabla(\mu)=H^0_G(\mu)$ and 
$H^0_G(\la)\simeq H^0_G(\mu) \otimes D^+(1, \ldots, m)^{a^+} \otimes D^-(m+1, \ldots, m+n)^{a^-}$ as $G_{ev}$-supermodules.
\end{pr}
\begin{proof}
Choose $a^+$ and $a^-$ such that $\mu_m\geq n$ and $\mu_{m+n}\geq 0$. Then $\nabla(\mu)=H^0_G(\mu)$ by \cite{zubcom} or by direct observation that $D^+(1, \ldots ,m)^n \Lambda(Y)$ is polynomial.
We have $H^0_{G_{ev}}(\mu)\simeq H^0_{G_{ev}}(\la)\otimes D^+(1, \ldots, m)^{a^+} \otimes D^-(m+1, \ldots, m+n)^{a^-}$ as $G_{ev}$-supermodules.
Since $H^0_G(\mu)\simeq H^0_{G_{ev}}(\mu)\otimes \Lambda(Y)$ and $H^0_G(\la)\simeq H^0_{G_{ev}}(\la)\otimes \Lambda(Y)$ as $G_{ev}$-supermodules, 
the claim follows.
\end{proof}
Therefore even-primitive vectors in $H^0_G(\la)$ are obtained by tensoring even-primitive vectors in $H^0_G(\mu)$ (described in Theorem \ref{LR}) with 
$D^+(1, \ldots, m)^{-a^+} \otimes D^-(m+1, \ldots, m+n)^{-a^-}$.
 
\section{The supermodule $\nabla(\la)$}

\subsection{Schur superalgebras}
The category of polynomial supermodules over $GL(m|n)$ is equivalent to the category of supermodules over the corresponding Schur superalgebra $S(m|n)$, see \cite{br} or Section 2.2 of \cite{muir}. Under this equivalence, the supermodule $\nabla_G(\la)$ is identified with the costandard module $\nabla(\la)$ of $S(m|n)$. 
Moreover, the category of $S(m|n)$-supermodules is semisimple; its simple supermodules $L_S(\la)$ are parametrized by $(m|n)$-hook partitions $\la$, and their characters are given by Hook Schur functions $HS_{\la}(x,y)$. The function $HS_{\la}(x,y)$ is given by summing certain monomials corresponding to $(m|n)$-semistandard tableaux. In particular, by 6.11 of \cite{br} there is
\begin{equation}\label{HS} HS_{\la}(x,y)=  \sum_{\mu<\la} \sum_{\nu} C^{\la'}_{\mu' \nu} S_{\mu}(x)S_{\nu}(y),\end{equation}
where $S_{\mu}(x)$ is the Schur function on variables $x_1, \ldots, x_m$ corresponding to $c_{11}, \ldots,$ $ c_{mm}$ and 
$S_{\nu}(x)$ is the Schur function on variables $y_1, \ldots, y_n$ correspoding to $c_{m+1, m+1},$ $\ldots,  c_{m+n,m+n}$, and 
$C^{\la'}_{\mu' \nu}$ is the Littlewood-Richardson coefficient.

It follows that the multiplicity of even-primitive vectors of weight $(\mu|\nu)$ in $\nabla(\la)$ is $C^{\la'}_{\mu' \nu}$.
For more details the reader is asked to consult \cite{br} or \cite{muir}.

\subsection{$G_{ev}$-primitive vectors in $\nabla(\la)$}

From now on we will assume that a $(m|n)$-hook weight $\la$ is fixed and $(I|J)$ is admissible of length $k$ and content $cont(I|J)=(\iota|\kappa)$, such that $\mu=(\mu_1, \ldots, \mu_m)=\la^+-\iota$ and $\nu=(\nu_{m+1}, \ldots, \nu_{m+n})$ $=\la^-+\kappa$ are dominant partitions.
Denote by $P(\mu|\nu)=P_k(\mu|\nu)$ the set of even-primitive vectors in $H^0_{G_{ev}}(\la)\otimes \wedge^k Y$ (viewed as a subspace of $H^0_G(\la)$) of weight $(\mu|\nu)$ that are linear combinations of elements $\overline{\pi}_{K|L}$ for admissible $(K|L)$
of content $(\iota|\kappa)$.

Previously we have constructed the even-primitive vector $\overline{v}(T^+)$ in $H^0_G(\la)$ corresponding to a tableaux $T^+$, such that its weight is $(\mu|\nu)$ as above.
Denote by $k$ the length of $T^+$, which is the cardinality of the diagram $\mathcal{D}^+$. Theorem \ref{p2} shows that the vector $\overline{v}(T^+)=v^+_T\overline{\rho}(\tau T^+)$ is an even-primitive vector of weight $(\mu|\nu)$ in $P_k(\mu|\nu)$.
Since this construction is based on a tableau $T^+$, the weight $(\mu|\nu)$ of the vector $\overline{v}(T^+)$ is polynomial. 
In this section we will show that the vectors $\overline{v}(T^+)$ for $T^+\in M((\la^+)'/\mu', \nu/\omega)$ form a basis of all $G_{ev}$-primitive vectors of weight $(\mu|\nu)$ 
in the supermodule $\nabla(\la)$.

If $\la_m<n$, then $H^0_G(\la)$ is not a polynomial supermodule, and there is an even-primitive vector in $H^0_G(\la)$ that cannot be of the form $\overline{v}(T^+)$. 
However, every even-primitive vector from $P(\mu|\nu)$ has polynomial weight $(\mu|\nu)$ and belongs to $\nabla(\la)$.

We have already addressed the question raised in the comments after Theorem 4.4. of \cite{fm} and confirmed that every even-primitive vector in $H^0_G(\la)$ is a linear combinations of elements $\overline{\pi}_{K|L}$. We generalize this statement for even-primitive vectors in $\nabla(\la)$ as follows.

\begin{tr}\label{con1}
Every even-primitive vector of $\nabla(\la)$ of weight $(\mu|\nu)$ belongs to $P(\mu|\nu)$. A basis of even-primitive vectors in $\nabla(\la)$ of the weight $(\mu|\nu)$ consists
of vectors $\overline{v}(T^+)$ for $T^+\in M((\la^+)'/\mu', \nu/\omega)$.
\end{tr}
\begin{proof}
Denote by $\tilde{\la}$ and $\tilde{\mu}$ the partitions such that $\tilde{\la}^+_i=\la^+_i +n$ and $\tilde{\mu}_i=\mu_i+n$ for every $1\leq i\leq m$, and $\tilde{\la}^-_j=\la^-_j$ for every $1\leq j\leq n$.
Then the skew shape $(\tilde{\la}^+)'/\tilde{\mu}'$ and its diagram $\tilde{\mathcal{D}}^+$ are obtained by shifting the skew shape $(\la^+)'/\mu'$ and its diagram $\mathcal{D}^+$
by $n$ rows downwards. Therefore, the cardinalities of the sets $M((\tilde{\la}^+)'/\tilde{\mu}', \nu/\omega)$ and $M((\la^+)'/\mu', \nu/\omega)$ coincide.

Since $\tilde{\la}_m\geq n$, by Proposition \ref{compare} the cardinalities of of the sets $LR(\tilde{\la}'/\tilde{\mu}',\nu)$ and $M((\tilde{\la}^+)'/\tilde{\mu}', \nu/\omega)$
are the same, and by Theorem \ref{LR} and Theorem 6.11 of \cite{br} they are equal to the multiplicity $C^{\tilde{\la}'}_{\tilde{\mu}',\nu}$ of $G_{ev}$-primitive vectors of weight $(\tilde{\mu},\nu)$ in $H^0_G(\tilde{\la})$.

Using the representation of induced supermodules $H^0_G(\la)$ as $H^0_{G_{ev}}(\la) \otimes \wedge(Y)$, we derive that tensoring an arbitrary even-primitive vector 
of weight $(\mu|\nu)$ in $H^0_G(\la)$ with $D^+(1, \ldots, m)^n$ gives an even-primitive vector of weight $(\tilde{\mu}|\nu)$ in 
$H^0_G(\tilde{\la})$, and vice versa, tensoring an arbitrary even-primitive vector of weight $(\tilde{\mu}|\nu)$ in 
$H^0_G(\tilde{\la})$ with $D^+(1, \ldots, m)^{-n}$ gives an even-primitive vector of weight $(\mu|\nu)$ in $H^0_G(\la)$.
Since the linear independence of vectors is preserved by tensoring, we conclude that the multiplicity of even-primitive vectors of weight $(\mu|\nu)$ in $H^0_G(\la)$ is the same as 
the multiplicity of even-primitive vectors of weight $(\tilde{\mu}|\nu)$ in $H^0_G(\tilde{\la})$.

Since the cardinality of $M((\la^+)'/\mu', \nu/\omega)$ is the same as the multiplicity of all even-primitive vectors of weight $(\mu|\nu)$ in $H^0_G(\la)$, and 
elements $\overline{v}(T^+)$ for $T^+\in M((\la^+)'/\mu', \nu/\omega)$ are linearly independent by Theorem \ref{jeden}, we conclude that
the vectors $\overline{v}(T^+)$ for $T^+\in M((\la^+)'/\mu', \nu/\omega)$ form a basis of all even-primitive vectors in $\nabla(\la)$ of the weight $(\mu|\nu)$.
\end{proof}

\begin{cor}
Even-primitive vectors in $\nabla(\la)$ are exactly those even-primitive vectors in $H^0_G(\la)$ of polynomial weight.
\end{cor}

\subsection{A connection of marked tableaux to pictures}

We will explain the connection of marked tableaux to pictures in the sense of Zelevinsky - see \cite{zel} and \cite{van1}. Our notation is a hybrid of \cite{zel} and \cite{van1}.

We will denote the set of partitions by $\mathcal{P}$ and for $\alpha, \beta,\kappa \in \mathcal{P}$ the Littlewood-Richardson coefficients by $C^{\alpha}_{\beta,\kappa}$.

\begin{df}
For a skew partition $\alpha/\beta$ define the partial orders $\leq_{\nwarrow}$ and $\leq_{\swarrow}$ on the entries $(i,j)$ of the diagram $[\alpha/\beta]$ as follows.

$(i,j)\leq_{\nwarrow} (i',j')$ if and only if $i\leq i'$ and $j\leq j'$;

$(i,j)\leq_{\swarrow} (i',j')$ if and only if $i\leq i'$ and $j\geq j'$.

We will also use the total ordering $\leq_r$ that refines $\leq_{\swarrow}$ and is defined as 
$(i,j)\leq_r (i',j')$ if and only if $i< i'$ or ($i=i'$ and $j\geq j'$).

\end{df}

\begin{df} Let $\alpha/\beta$ and $\gamma/\delta$ be skew partitions, and $f:\alpha/\beta \to \gamma/\delta$ is a map.
A map $f$ is called a picture from $\alpha/\beta$ to $\gamma/\delta$ if
\begin{itemize}
\item $f$ is a bijection
\item If $x \leq_{\nwarrow} y$ then $f(x)\leq_{\swarrow} f(y)$
\item If $f(x)\leq_{\nwarrow} f(y)$ then $x\leq_{\swarrow} y$.
\end{itemize}
The set of pictures from $\alpha/\beta$ to $\gamma/\delta$ is denoted by $Pict(\alpha/\beta,\gamma/\delta)$.
\end{df}

\begin{df}
Let $f:\alpha/\beta \to \gamma/\delta$ be a picture. Define the row reading of $f$ to be the tableau $E^+$ of shape $\alpha/\beta$ such that the entry at the position 
$(i,j)$ is the first coordinate of $f(i,j)$, and define the column reading of $f$ to be the tableau $E^-$ of shape $\gamma/\delta$ such that the entry at the position 
$(i,j)$ is the second coordinate of $f^{-1}(i,j)$.
\end{df}

It is immediate from the definitions that the row reading $E^+$ of a picture $f$ is a semistandard tableau and the columm reading $E^-$ is an anti-semistandard tableau.

\begin{ex}
Let $\alpha=(3,3,2,1)$, $\beta=(2,0,0,0)$, $\gamma=(4,2,2,1)$ and $\delta=(1,1,0,0)$. Let the picture $f:\alpha/\beta \to \gamma/\delta$ be given as
\[\begin{array}{ccc}
&&f\\
a&d&g\\
b&e&\\
c&&
\end{array}\mapsto 
\begin{array}{cccc}
&f&d&a\\
&g&&\\
e&b&&\\
c&&&
\end{array}.\]
Then the row and collumn readings of $f$ are given as 
\[E^+=\begin{array}{ccc}
&&1\\
1&1&2\\
3&3&\\
4&&
\end{array}\text{ and } 
E^-=\begin{array}{cccc}
&3&2&1\\
&3&&\\
2&1&&\\
1&&&
\end{array}.\]
\end{ex}

\begin{lm}\label{van}
A tableau $E$ of shape $\alpha/\beta$ is a row reading tableau of a picture $f:\alpha/\beta \to \gamma/\delta$ if and only if
$E$ is semistandard and $E$ is shifted Yamanouchi.
\end{lm}
\begin{proof}
The proof is obtained by adjusting arguments in the proof of Proposition 2.6.1 of \cite{van1} and the remark following it.
Here we need to keep in mind that the definitions of pictures and the partial order $\leq_{\swarrow}$ in \cite{van1} differs from ours (which conform to those given in \cite{zel}) by transposition of domain an image side - see the footnote on page 325 of \cite{van1}.
\end{proof}

\begin{pr}\label{corresp}
There is a bijective correspondence between pictures $f:(\la^+)'/\mu' \to \nu/\omega$ and marked tableau $T^+\in M((\la^+)'/\mu',\nu/\omega)$. Under this correspondence the
tableau $T^+$ is the row reading of $f$, and vice versa, $f$ is given by the repositioning maps $Rpos:\mathcal{D}^+ \to \mathcal{D}^-$ corresponding to the tableau $T^+$.
\end{pr}
\begin{proof}
By Lemma \ref{van} we know that pictures $f:(\la^+)'/\mu' \to \nu/\omega$ correspond to tableau $T^+$ of the shape $(\la^+)'/\mu'$, which are semistandard and shifted Yamanouchi. 

Using Proposition \ref{symmetry}, we conclude that such $T^+$ is the row ordering of a picture $f:(\la^+)'/\mu' \to \nu/\omega$ if and only if
$T^+\in M((\la^+)'/\mu',\nu/\omega)$.
\end{proof}

Note that Proposition \ref{symmetry} allows us to characterize a picture $f:(\la^+)'/\mu' \to \nu/\omega$ in a completely symmetric way, using marked tableau $T^+$. 
The symmetry is given by the requirement that $T^+$ is semistandard and $T^-$ is anti-semistandard. Usually pictures are characterized using lattice permutations 
(see Proposition 2.6.1 of \cite{van1} and the remark following it) and the description does not seem symmetric. With our approach, we are using a pair of tableaux - semistandard $T^+$
and anti-semistandard $T^-$ connecting using the repositioning map $Rp$. It is the repositioning map $Rp$ that provides the connection betwen $T^+$ and $T^-$, namely $Rp(T^+)=T^-$, 
and the specific definition of $Rp$ we are using is related to the lattice condition. 

An immediate consequence of Proposition \ref{corresp} is the following result.

\begin{pr}
In the notation as above we have 
\[C^{\la'}_{\mu' \nu}=\sum_{\kappa\in \mathcal{P}} C^{(\la^+)'}_{\mu',\kappa} C^{\nu}_{\omega, \kappa}.\]
\end{pr}
\begin{proof}
By Theorem \ref{con1}, the cardinality of $M((\la^+)'/\mu',\nu/\omega)$ equals the multiplicity of even-primitive vectors of weight 
$(\mu|\nu)$ in $\nabla(\la)$. It follows from (\ref{HS}) that this multiplicity equals $C^{\la'}_{\mu' \nu}$.

On the other hand, the cardinality of the set $Pict((\la^+)'/ \mu',\nu/\omega)$ equals
\[\sum_{\kappa\in \mathcal{P}} C^{(\la^+)'}_{\mu',\kappa} C^{\nu}_{\omega, \kappa}\] by Theorem 1 of \cite{zel}.

Proposition \ref{corresp} concludes the proof.
\end{proof}

\end{document}